\documentclass{article}
\usepackage{amsmath}
\usepackage{amsfonts}
\usepackage{latexsym}
\usepackage{amssymb}

\newtheorem{thm}{Theorem}[section]
\newtheorem{theorem}[thm]{Theorem}

\newtheorem{lemma}[thm]{Lemma}
\newtheorem{prop}[thm]{Proposition}

\newtheorem{definition}[thm]{Definition}

\newtheorem{remark}[thm]{Remark}

\newcommand{\beq}{\begin{equation}}
\newcommand{\eeq}{\end{equation}}
\newcommand{\beqa}{\begin{eqnarray}}
\newcommand{\eeqa}{\end{eqnarray}}
\newcommand{\beqas}{\begin{eqnarray*}}
\newcommand{\eeqas}{\end{eqnarray*}}
\newcommand{\bi}{\begin{itemize}}
\newcommand{\ei}{\end{itemize}}

\newcommand{\vgap}{\vspace{.1in}}

\def\endproof{{\ \hfill\hbox{%
      \vrule width1.0ex height1.0ex
    }\parfillskip 0pt}\par}

\newcommand{\R}{\mathbb{R}}

\newcommand{\lam}{{\lambda}}

\newcommand{\inner}[2]{\langle #1,#2\rangle}

\newcommand{\cl}{\mathrm{cl}\,}

\newcommand{\argmin}{\mathrm{argmin}\,}

\newcommand{\dom}{\mathrm{dom}\,}

\newcommand{\Argmin}{\mathrm{Argmin}\,}

\newcommand{\Conv}[1]{\mbox{\rm Conv}(\R^{#1})}

\newcommand{\bConv}[1]{\overline{\mbox{\rm Conv}}\,(\R^{#1})}

\usepackage{url}
\usepackage{graphicx}
\usepackage{diagbox}
\usepackage[utf8]{inputenc}
\usepackage{color}
\usepackage{makecell}
\usepackage{tablefootnote}
\usepackage{algpseudocode}
\usepackage{enumerate}
\usepackage{algorithm}

\newcommand*{\E}{\mathbb{E}}
\usepackage{multirow}

\begin{document}
\title{Parameter-free proximal bundle methods
with 
adaptive \\stepsizes
	for hybrid convex composite optimization problems}
\date{October 27, 2024   \\
\begin{center}
1st revision: Aug 24, 2025 \\
2nd revision: Apr 28, 2026 \\
3rd revision: May 21, 2026 \\
4th revision: Oct 11, 2026
    \end{center}}
\author{		
        Renato D.C. Monteiro \thanks{School of Industrial and Systems
		Engineering, Georgia Institute of
		Technology, Atlanta, GA, 30332-0205.
		(email: {\tt rm88@gatech.edu} and {\tt hzhang906@gatech.edu}). This work
			was partially supported by AFOSR Grants FA9550-22-1-0088 and
            FA9550-25-1-0182.}\qquad 
	 Honghao Zhang 
   \footnotemark[1]
	}

 \maketitle
\begin{abstract}
   This paper develops  parameter-free and line-search-free adaptive  proximal bundle methods with two important features:
1) adaptive choice
of variable prox stepsizes that  "closely fit" the instance under consideration; and
2) adaptive criterion for making the occurrence of serious steps easier.
Computational experiments show that our methods perform substantially fewer consecutive null steps (i.e., a shorter cycle)
while maintaining the number of serious steps under control.  As a result,
our methods perform significantly
fewer iterations than
its counterparts
based on a constant prox stepsize choice and a non-adaptive cycle termination criterion.
Moreover, our methods are very
 robust relative to the initial stepsize. 

 \vgap


\vgap
		
  
\end{abstract}
\section{Introduction}
	 Let $ f, h: \R^{n} \rightarrow \R\cup \{ +\infty \} $ be proper lower semi-continuous  convex functions such that
	$ \dom h \subseteq \dom f $ and
	consider
	the optimization problem
	\begin{equation}\label{eq:ProbIntro}
	\phi_{*}:=\min \left\{\phi(x):=f(x)+h(x): x \in \R^n\right\}.
	\end{equation}
	It is said that \eqref{eq:ProbIntro} is a  hybrid convex composite optimization (HCCO) problem if there exist
	nonnegative scalars $M_f,L_f$ 
	and a
	first-order oracle $f':\dom h \to \R^n$
	(i.e., $f'(x)\in \partial f(x)$
	for every $x \in \dom h$)
	satisfying 
	$\|f'(u)-f'(v)\| \le 2M_f +L_f\|u-v\|$ for every $u,v \in \dom h$. The main goal of this paper
	is to study the complexity of
	two parameter-free and line-search-free adaptive proximal bundle methods  for solving the HCCO problem \eqref{eq:ProbIntro}.


  {\it Literature review on proximal bundle (PB) methods:}
PB methods solve a sequence of
prox bundle subproblems
\begin{equation}\label{def:xj1}
	    x_j = \underset{u\in \R^n}\argmin \left \{\Gamma_j (u) + \frac{1}{2\lam_j} \|u-x^c \|^2 \right\},
	\end{equation}
 where $\Gamma_j$ is a bundle approximation of $\phi$ (i.e., a simple convex function underneath $\phi$) and
 $x^c$ is the current prox center.
 The prox center is updated to $x_j$ (i.e., a serious step is performed) only when the pair $(x_j,\lam_j)$ satisfies a certain error criterion;
 otherwise, the prox center is kept the same
 (i.e., a null step is performed). Regardless of the step performed, the bundle $\Gamma_j$ is updated to account for the newest iterate $x_j$.
 In the discussion below, a sequence of consecutive null steps followed
 by a serious step is referred to as a cycle.  Classical PB methods (see e.g. \cite{diaz2023optimal,du2017rate,kiwiel2000efficiency, lemarechal1975extension,lemarechal1978nonsmooth,Mifflin1982,wolfe1975method})
 perform the serious step
 when $x_j$ satisfies a
relaxed descent condition
(e.g., see the paragraph containing  equation (15) in \cite{liang2021proximal}), which in its unrelaxed form implies  that $\phi(x_j) \le \phi(x^c)$.
On the other hand, a new class of
PB methods (see e.g. \cite{guigues2026universal,liang2021proximal,liang2024unified,liang2023proximal})
perform the serious step
 when  the best $\phi$-valued iterate $x_j$, say $y_j$,
 satisfies $\phi(y_j)-m_j \le \delta$ where
 $m_j$ is the optimal value of \eqref{def:xj1} and 
 $\delta$ is a suitably chosen tolerance.
Although $y_j$  does
 not necessarily satisfy the descent condition,
 it does satisfy a  $\delta$-relaxed version of it.
 It is shown in \cite{liang2021proximal,liang2024unified} that if
 $\lam>0$ is such that
 $\max\{ \lam , \lam^{-1} \} = {\cal O}(\varepsilon^{-1})$,
 then the
 new class of  PB methods with $\lam_j=\lam $ for every $j$ achieve an $\tilde {\cal O}(\varepsilon^{-2})$ 
 iteration complexity to obtain an
 $\varepsilon$-solution
 regardless of whether 
 $\dom h$ is bounded or not.
 In contrast, papers \cite{diaz2023optimal,kiwiel2000efficiency}
 show that  classical 
 PB methods achieve:
 i) 
an ${\cal O}(\varepsilon^{-3})$
iteration complexity under
the assumption that
$\lam = \Theta(1)$ regardless of whether $\dom h$ is bounded or not;
and ii)
an ${\cal O}(\varepsilon^{-2})$
iteration complexity under
the assumption that
$\lam = \Theta(\varepsilon^{-1})$ for the case where $\dom h$ is bounded.


Despite their strong theoretical guarantees, the practical performance of the proximal bundle methods is highly sensitive to the choice of stepsize. Large proximal stepsizes can lead to long cycles, making the development of adaptive strategies to update the stepsizes particularly important. A natural strategy, upon observing a long cycle, is to reduce the proximal stepsize at the end of it before initiating a new one. However, a key drawback of this approach is its delay corrective action - it takes place at the end of a (possibly long) cycle.
 

Several papers of the classical PB methods (see e.g. \cite{bonnans1995family,de2016doubly,diaz2023optimal,karmitsa2010adaptive,doi:10.1137/040603929,lemarechal1997variable} of which only \cite{diaz2023optimal} deals with complexity analysis) have proposed ways of generating variable (but, nonadaptive) prox stepsizes to improve 
their computational performance. 
 More recently, \cite{guigues2026universal} developed a new variant of PB method for solving either the convex or strongly convex version of \eqref{eq:ProbIntro} which: requires no knowledge of the Lipschitz parameters $(M_f,L_f)$ and the strong convexity parameter $\mu$ of $\phi$; and
 allows the stepsize to change only at the beginning of each cycle.
 A potential drawback of the method of \cite{guigues2026universal} is that
 it can restart a cycle with its current initial prox stepsize $\lam$ divided by two if $\lam$ is found to be large, i.e., the method can backtrack.



{\it Main Contribution.} The goal of this paper is to develop  parameter-free and line-search-free adaptive  PB methods,
namely Ad-GPB* and Ad-GPB, with two important features:
1) adaptive choice
of variable prox stepsizes that  "closely fit" the instance under consideration; and
2) an adaptive criterion for making the occurrence of serious steps easier. The next two paragraphs further elaborate on these two features.


A key observation underlying the  PB methods is that the cycle length can be controlled by ensuring that a key inequality holds at every null iteration. Motivated by this observation, our method employs the adaptive stepsize strategy  which halves the stepsize after a null iteration that violates the key inequality and immediately proceeds to the next one. Because this inequality is parameter-free, the resulting procedure is line-search-free and does not require any knowledge of the parameters $(M_f,L_f)$. More importantly, it generates cycles that are substantially shorter than those produced by its closest counterparts in \cite{liang2021proximal,liang2024unified}.

To motivate  the relaxed cycle termination rule mentioned in 2), we briefly compare the theoretical cycle-length behavior of the classical \cite{diaz2023optimal,du2017rate,kiwiel2000efficiency, lemarechal1975extension,lemarechal1978nonsmooth,Mifflin1982,wolfe1975method} and the new \cite{guigues2026universal,liang2021proximal,liang2024unified}  PB methods.
In theory, for most choices of prox stepsizes, classical  PB methods may perform a total of ${\cal O}(\varepsilon^{-3})$  null iterations while the new  PB methods perform at most ${\cal O}(\varepsilon^{-2})$
consecutive null iterations.
This improvement comes from the more permissive $\delta$-criterion used by the new class of PB methods to terminate a cycle. Our Ad-GPB* method pursues the idea of further relaxing the cycle termination criterion to reduce its overall number of iterations, and hence improve its computational performance while retaining all the theoretical guarantees of the new class of PB methods. More specifically, under the simplifying assumption that $\phi_*$ is known,
an Ad-GPB* cycle stops when
the inequality
$\phi(y_j)-m_j \le \delta +  [ \phi(y_j)-\phi_*]/2$ 
is satisfied.
The addition of the (usually large) term $[ \phi(y_j)-\phi_*]/2$ relaxes the cycle termination criterion,
thereby causing Ad-GPB* to perform shorter cycles.  In addition, when the objective function \(\phi\) is $\mu$-strongly convex, it is shown that a restart version of Ad-GPB* (which does not use $\mu$) has a better complexity bound in terms of $\varepsilon$, namely,
$\tilde {\cal O}(1/(\mu \varepsilon))$.

We also describe another parameter-free and line-search-free adaptive PB variant, namely, Ad-GPB,
which does not assume that $\phi_*$ is known and
 shares the same two features as Ad-GPB* outlined in the two paragraphs above.
 At the expense of assuming that the domain of $h$ is bounded, this variant instead generates a sequence $\{\hat \ell_k\}$ of
 lower bounds for $\phi_*$ and uses the primal-dual gap $\phi(y_j) - \hat \ell_k$ to terminate.
 

Finally, computational experiments show that Ad-GPB* performs substantially fewer consecutive null steps 
while maintaining the number of serious steps under control.  As a result,
Ad-GPB* performs significantly
fewer iterations than the GPB method of \cite{guigues2026universal,liang2021proximal,liang2024unified}.
Moreover, in contrast to GPB, Ad-GPB* is
 robust to initial stepsize choices.
 In addition, we use Ad-GPB* to solve a relevant class of Lagrangian cut subproblems that arise while solving binary stochastic integer programming problems.

 {\it Other related methods:}  Another method related to, and subsequently developed from, proximal bundle methods is the bundle-level method, which was first proposed in \cite{lemarechal1995new} and later extended in various ways (see, e.g., \cite{ben2005non, kiwiel1995proximal, lan2015bundle}). Finally, paper \cite{de2016doubly} presents a doubly stabilized bundle method  whose prox subproblems combine elements from both proximal bundle and bundle-level methods.

     {\it Organization of the paper.}  
\S1.1 presents basic definitions and notation used throughout the paper.
\S2 contains two subsections. \S2.1 formally describes problem~(1) and
the assumptions made on it, while \S2.2 reviews the GPB method and
presents the motivation of this paper.
\S3 contains two subsections. \S3.1 presents the Ad-GPB* framework and
states its main iteration-complexity results, namely, Theorem~3.2 and
Proposition~3.3. \S3.2 presents two special bundle update schemes.
\S4 contains two subsections. \S4.1 provides a bound on the number of
iterations within a cycle of Ad-GPB*. \S4.2 establishes bounds on the
number of cycles and the total number of iterations performed by
Ad-GPB*, and proves Theorem~3.2 and Proposition~3.3.
\S5 presents the Ad-GPB framework under the assumption that $\phi^*$ is
unknown and states its main iteration-complexity result, along with its
proof. \S6 presents the complexity of a restart version of Ad-GPB*
under the additional assumption that the objective function $\phi$ is
strongly convex. \S7 presents the numerical results comparing Ad-GPB*
and Ad-GPB with the two-cut bundle update scheme against other bundle
methods. Finally, \S8 presents the concluding remarks.

 \subsection{Basic definitions and notation} \label{subsec:DefNot}

    The sets of real numbers and positive real numbers are denoted by $\R$ and $\R_{++}$, respectively. 
	Let $\R^n$ denote the standard $n$-dimensional Euclidean space equipped with  inner product and norm denoted by $\left\langle \cdot,\cdot\right\rangle $
	and $\|\cdot\|$, respectively. 
	Let $\log_2(\cdot)$ denote the logarithm to base $2$, and define 
$\log_2^+(\cdot) := \max\{\log_2(\cdot), 0\}$. 
Similarly, let $\ln(\cdot)$ denote the natural logarithm, and define 
$\ln^+(\cdot) := \max\{\ln(\cdot), 0\}$. Let ${\cal O}$ denote the standard big-O notation. 

	For a given function $\varphi: \R^n\rightarrow (-\infty,+\infty]$, let $\dom \varphi:=\{x \in \R^n: \varphi (x) <\infty\}$ denote the effective domain of $\varphi$ 
	and $\varphi$ is proper if $\dom \varphi \ne \emptyset$.
	 A proper function $\varphi: \R^n\rightarrow (-\infty,+\infty]$ is
$\mu$-convex, for some $\mu\geq 0$, if
\[
\varphi(\alpha x+(1-\alpha)y)
\leq
\alpha\varphi(x)+(1-\alpha)\varphi(y)
-\frac{\mu}{2}\alpha(1-\alpha)\|x-y\|^2
\]
for every $x,y\in\dom\varphi$ and $\alpha\in[0,1]$.
 Denote the set of all proper lower semicontinuous convex functions by $\bConv{n}$. $\mathbf{1}$ denotes the all ones vector in $\R^n$.
 
	The subdifferential of $ \varphi $ at $x \in \dom \varphi$ is denoted by
	\beq \label{def:subdif}
 \partial \varphi (x):=\left\{ s \in\R^n: \varphi(y)\geq \varphi(x)+\left\langle s,y-x\right\rangle, \forall y\in\R^n\right\}.
        \eeq
 The set of proper closed convex functions $\Gamma$ such that $\Gamma \le \phi$ is denoted 
 by ${\cal B}(\phi)$ and any such $\Gamma$ is called a bundle for $\phi$.

Lastly, we adopt the convention that
\[
a/0:=+\infty \qquad \text{for every } a>0.
\]
 
\section{Main problem and motivation}\label{mainpr}
This section contains two subsections. The first one describes the main problem and corresponding assumptions. The second one presents the motivation of Ad-GPB*.
\subsection{Main problem}\label{Sec:Ass}
The problem of interest in this paper is \eqref{eq:ProbIntro} which is
	assumed to satisfy the following conditions
	for some
	constants $M_f \ge 0$ and $L_f \ge 0 $:
	\begin{itemize}
 \item[(A1)]
		 $h \in \bConv{n}$ and
		there exists $D \ge 0$ such that
    \begin{equation}\label{defD}
    \sup_{x,y \in \dom h} \|y-x\|\le D;
    \end{equation}
		\item[(A2)]
		$f \in \bConv{n}$ is such that
		$\dom h \subset \dom f$, and a subgradient oracle,
		 i.e.,
		a function $f':\dom h \to \R^n$
		satisfying $f'(x) \in \partial f(x)$ for every $x \in \dom h$, is available;
		\item[(A3)]
		for every $x,y \in \dom h$,
		\[
		\|f'(x)-f'(y)\| \le 2M_f+L_f\|x-y\| .
		\]
	\end{itemize}
 \color{black} 
 In addition to the above assumptions, it is also assumed that $h$ is simple in the sense that,
for any $\lam>0$ and
affine function ${\cal A}$, the following two optimization problems 
\begin{equation}\label{easyso}
\min_u {\cal A} (u) + h(u), \qquad \min_u {\cal A}(u) + h(u) + \frac{1}{2\lam}\|u\|^2
\end{equation}
are easy to solve. 

We now make some remarks about assumptions (A1)-(A3). First, it can be shown that (A1) implies that both problems in \eqref{easyso} have optimal solutions. Second, it can also be shown that (A1) implies that the set of optimal solutions $X^*$ of
		problem \eqref{eq:ProbIntro} is nonempty. Third, Assumption (A3) is used to model heterogeneous convex problems that contain both nonsmooth and smooth components. Indeed, when $L_f=0$, (A3) reduces to a bounded-subgradient-type condition commonly used in nonsmooth convex optimization, while when $M_f=0$, it reduces to the standard Lipschitz-gradient condition for smooth convex optimization. Similar hybrid conditions have been used in the study of hybrid convex composite optimization and related bundle-type methods (see, e.g., \cite{liang2021proximal,liang2024unified}). 
Fourth,
 letting
 $\tilde \ell_f(\cdot;x)$
 denote the linearization of $f$ at $x$, i.e., 
    \begin{equation}\label{def:gamma}
 \tilde \ell_f(\cdot;x) := f(x)+\inner{f'(x)}{\cdot-x} \quad \forall x\in \dom h,
	\end{equation}
  then it is well-known that (A3) implies that for every $x,y \in \dom h$,	\begin{equation}\label{ineq:est}
	f(x)- \tilde \ell_f(x;y) \le 2M_f \|x-y\| + \frac {L_f}2\|x-y\|^2 .
	\end{equation}
Fifth, define the composite linearization of the objective $\phi$  of \eqref{eq:ProbIntro} at $x$ as
 \begin{equation}\label{def:ellphi}
     \ell_{\phi}(\cdot;x) := \tilde \ell_f(\cdot;x)+h(\cdot)\quad \forall x\in \dom h.
 \end{equation}
Finally, for any \(z \in \dom h\), let \(d(z)\) denote the distance of \(z\) to the set of optimal solutions \(X^*\), i.e.,
\begin{equation}\label{def:d0}
    d(z) := \|z-x^*(z)\|,
    \qquad
    x^*(z) \in \argmin\{\|z-x\|: x\in X^*\}.
    \end{equation}
\subsection{Motivation for this work}\label{GPB}
The main purpose of this subsection is to motivate the variable-stepsize proximal bundle methods presented in this paper in light of the constant stepsize  proximal bundle  method of \cite{liang2024unified}.

We start by outlining the PB
method of \cite{liang2024unified}, referred to as GPB (as in \cite{liang2024unified}) from now on.
\noindent\rule[0.5ex]{1\columnwidth}{1pt}
	
	GPB 
	
	\noindent\rule[0.5ex]{1\columnwidth}{1pt}
    $\textbf{Input:}$
  $(x_0,\varepsilon)\in \R^n\times\R_{++}$, $\lam \in \R_{++}$
    \begin{itemize}
		\item [0.]  
  find $\Gamma_1 \in {\cal B}(\phi)$ such that 
  $\Gamma_1 \ge \ell_{\phi}(\cdot;  x
  _0)$ and set
	$y_0= x^c=x_0$, and $j=1$; 
	    \item[1.] compute the optimal solution $x_j$ and optimal value $m_j$ of  
        \begin{equation}
         \underset{u\in \R^n}\argmin \left \{\Gamma_j (u) + \frac{1}{2\lam} \|u-x^c \|^2 \right\}\label{def:xj}, \\
        \end{equation}
         and set $      
     y_j := \argmin
     \left \{\phi(x) : x \in \{y_{j-1},x_j\} \right\}$ 
     and
     $t_j := \phi(y_j) - m_{j};  $
		\item[2.] 
        \begin{itemize}
            \item [a)] {\textbf{if}} $t_j \le \varepsilon/2$ is violated  {\textbf{then}} use $x_j$ to update $\Gamma_j$ to a new bundle $\Gamma_{j+1} \in {\cal B}(\phi)$;

        \item [b)] {\textbf{else}} compute
$\Gamma_{j+1} \in {\cal B}(\phi)$ such that
			$\Gamma_{j+1}\ge \ell_{\phi}(\cdot; x_j)$ and set
        $x^c \leftarrow x_j$; 
        
           
\end{itemize}
         	\item[3.] set  $j\leftarrow j+1$ and go to step 1.
	\end{itemize}
	\rule[0.5ex]{1\columnwidth}{1pt}

We now provide several remarks on GPB method. First, an iteration $j$ is referred to as a \emph{null iteration} if step 2.a) is executed; otherwise, it is called a \emph{serious iteration}. Second, a \emph{cycle} is defined as a sequence of consecutive null iterations occurring between two consecutive serious iterations. 
Third, even though the above GPB outline does not specify how the point $x_j$ is used to update the bundle $\Gamma_j$ in step 2.a), the details of this update procedure are discussed in \S\ref{sec:shadow}. Fourth,
it is shown in \cite{liang2024unified} that 
$t_j - \varepsilon/4$ converges geometrically. Specifically, it is proven that if 
$j$ is a  null iteration, then  $
t_{j} - \varepsilon/4 \le \tau_\lambda \left( t_{j-1} - \varepsilon/4 \right)$,
where $\tau_\lambda$ is the unique positive scalar satisfying
\begin{equation}\label{rel:tau1} 
    \frac{\tau_\lambda}{1 - \tau_\lambda} = \frac{8\lambda (M_f^2 + L_f\varepsilon)}{\varepsilon}.
\end{equation}
As a result, it is shown that the number of consecutive null iterations is 
$\tilde{\mathcal{O}}\left(1/(1 - \tau_\lambda) \right)$. 
If the prox stepsize $\lambda$ of GPB is large, and hence $\tau_\lambda$ is close to $1$,  the geometric convergence becomes slow,  which may cause GPB to generate long cycles, 
and hence perform poorly in practice.

   {\it Motivation for our work:} Our goal  is to develop a parameter and line-search free adaptive variant of GPB method, called Ad-GPB*, which chooses the prox stepsize more efficiently and generates shorter cycles than GPB.
   Its two main ideas are: i)  its prox stepsizes are allowed to change within a cycle and each inner iteration adaptively chooses the prox stepsize for the next iteration
    based on a key test inequality so that no backtracking and line search is performed; and
    ii)  its cycle stopping criterion is relaxed, thereby resulting in short cycles while still ensuring convergence of Ad-GPB*.
    Below, we elaborate on each of these two ideas.
\begin{itemize}
    \item [i)]A natural strategy for avoiding long cycles is to choose $\lam$ such that $\tau_\lambda \le \tau$ for some fixed $\tau \in (0,1)$ not too close to one (i.e., $\tau=0.95$). However, this strategy requires prior knowledge of the parameters $M_f$ and $L_f$, which are generally not known.
We now discuss two parameter-free strategies for prox stepsize selection.
These strategies  are motivated by the fact that the condition $\tau_\lam \le \tau$ implies that 
\begin{equation}\label{ineq:keyrecusrive}
    t_{j} - \frac{\varepsilon}{4} \le \tau \left( t_{j-1} - \frac{\varepsilon}{4} \right)
\end{equation}
for every null iteration $j$. 
Both strategies allow $\lam$ to change within a cycle in an attempt to either make it satisfy or become closer to satisfy \eqref{ineq:keyrecusrive}.
Specifically, in a null iteration $j$, both strategies compute $x_j$ and $m_j$ using \eqref{def:xj}
with $\lam$ replaced by $\lam_j$, and then $(y_j,t_j)$  as in step 1 of GPB. Hence, verification of \eqref{ineq:keyrecusrive} for some $\lam_{j}$ requires the computation of the exact solution of a subproblem as in \eqref{def:xj}.

The first and most obvious  strategy is to perform a line search that halves $\lambda_j$ and stops when \eqref{ineq:keyrecusrive} is satisfied. However, this strategy has  two drawbacks: i) it throws away the $\lam_j$'s, and hence the corresponding $x_j$'s, that violate \eqref{ineq:keyrecusrive}; and, ii) it is not forgiving in that it aggressively reduces the prox stepsize.

 The second, more forgiving, strategy, which is the one studied in this paper, follows a \emph{delayed update strategy}: given a pair  $(\Gamma_{j},\lambda_j)$, the iterate $x_j$ (and the next bundle $\Gamma_{j+1}$) is obtained by performing a GPB iteration with $\lam = \lam_j$ regardless of whether \eqref{ineq:keyrecusrive} is satisfied or not, and the next stepsize $\lam_{j+1}$ is set to $ \lambda_j $ if \eqref{ineq:keyrecusrive} is satisfied, or  $\lambda_j/2 $ otherwise.
Thus, instead of enforcing \eqref{ineq:keyrecusrive} immediately at iteration $j$, the rationale of this strategy is to slowly "approach" \eqref{ineq:keyrecusrive}. Its key advantage is that it \emph{never backtracks} and is therefore line-search free.
\item [ii)]  The cycle termination inequality in step~2.a) of GPB is relaxed by introducing an additional positive term to its right hand side.
 Specifically, if $\phi_*$ is known, a cycle of our new proposed variant of GPB stops when
the inequality
$\phi(y_j)-m_j \le \delta +  [ \phi(y_j)-\phi_*]/2$ 
is satisfied.
The addition of the (usually large) term $ [ \phi(y_j)-\phi_*]/2$ makes this inequality easier to satisfy, thereby resulting short cycles.
When
$\phi_*$ is not known, a strategy  similar to the one above is developed where   $\phi _*$ is replaced by a suitable lower bound.

\end{itemize}

\section{Main Algorithm}\label{Algorithm}

This section contains two subsections. 
The first one presents a $(M_f,L_f)$-parameter-free and line-search-free   PB method that requires the optimal value $\phi_*$ of \eqref{eq:ProbIntro}, namely, Ad-GPB*, and states its main complexity results. The second one presents two special bundle update schemes. 

 It is assumed throughout this section and the next one
 that  
\begin{equation}\label{keyass}
\boxed{\mbox{the optimal value $\phi_*$  of \eqref{eq:ProbIntro} is known.}}
\end{equation}

\subsection{Main Algorithm}\label{main:Alg}

This subsection presents  Ad-GPB* and its main complexity results.

We first give an outline of Ad-GPB*.
Ad-GPB* is an inexact proximal point method which,
given the $(k-1)$-th prox center $\hat{x}_{k-1} \in \mathbb{R}^n$,
finds a pair $(\hat x_k,\hat \lam_k)$ of prox stepsize $\hat \lam_k>0$ and  $k$-th prox-center $\hat x_k$ satisfying
a suitable error criterion
for being 
an approximate solution of the prox subproblem
\begin{equation}\label{subpro}
\hat{x}_k \approx \underset{u \in \mathbb{R}^n}{\operatorname{argmin}}\left\{\phi(u)+\frac{1}{2 \hat \lambda_k}\left\|u-\hat{x}_{k-1}\right\|^2\right\}.
\end{equation}
More specifically,
Ad-GPB* solves a sequence of
prox bundle subproblems
\begin{equation}\label{def:xj-1}
	    x_j = \underset{u\in \R^n}\argmin \left \{\Gamma_j (u) + \frac{1}{2\lam_j} \|u-\hat x_{k-1} \|^2 \right\},
	\end{equation}
 where $\Gamma_j$ is a bundle approximation of $\phi$  and $\lam_j \le \lam_{j-1}$ is an adaptively chosen prox stepsize,
 until the pair $(\hat x_k,\hat \lam_k)=(x_j,\lam_j)$ satisfy the approximate error criterion for \eqref{subpro}.
 In contrast to the GPB method of \cite{liang2024unified}, which can also be viewed in the setting outlined above,
Ad-GPB*: i) (adaptively) changes  $\lam_j$ while computing the next prox center $\hat x_k$; and, 
ii) Ad-GPB* stops the search for the next prox center $\hat x_k$ using a termination criterion based not only on the user-provided tolerance (the quantity $\varepsilon$ in the description below) as   GPB also does, but also on a suitable primal gap for \eqref{eq:ProbIntro}, a feature that considerably speeds up the computation of  $\hat x_k$ for many subproblems \eqref{subpro}.

We now formally describe Ad-GPB*.  Its description uses the definition of the set of bundles ${\cal B}(\phi)$ for the function $\phi$
     given in \S \ref{subsec:DefNot} and the following notion of shadow function.

\begin{definition}\label{def:shadow}
     Function $\bar \Gamma_j(\cdot) \in   {\cal B}(\phi)$ is called a shadow of $\Gamma_j(\cdot)$ for \eqref{def:xj1} if it satisfies
\[
\bar \Gamma_j\left(x_j\right)=\Gamma_j\left(x_j\right), \quad 
x_j = \underset{u\in \R^n}\argmin \left \{\bar \Gamma_j (u) + \frac{1}{2\lam_j} \|u-x^c\|^2 \right\},
\]
where $x_j$ is the optimal solution of \eqref{def:xj1}.
\end{definition}

We are now ready to state the Ad-GPB* algorithm.

\noindent\rule[0.5ex]{1\columnwidth}{1pt}
	
	$\hat x\, =\, $Ad-GPB*$(\hat x_0,\varepsilon,\hat \lam_0,\tau,\beta)$
	
	\noindent\rule[0.5ex]{1\columnwidth}{1pt}
    \textbf{Input:} $(\hat x_0,\varepsilon) \in \dom h \times  \R_{++}$, $(\hat \lam_0,\tau,\beta) \in \R_{++} \times (0,1) \times [0,1/4];$
    
    \noindent
    \textbf{Output:} $\hat x$ such that $\phi(\hat x)-\phi_* \le \varepsilon$;
	\begin{itemize}
		\item [0.] 
 set
	$y_0 =  \hat x_0 = \hat y_0 $, $j_0=0$, $j=k=1$;
\item [1.] \textbf{(stopping criterion)} {\textbf{if}}           $
    \phi(\hat y_{k-1})-\phi_* \le \varepsilon$,
  \textbf{then}  output $ \hat x = \hat y_{k-1}$, and {\bf stop}; 
    \item [2.] {\bf (initializing the $k$-th cycle)} choose $\lam_j \in [\hat \lam_{k-1}, 2 \hat \lam_{k-1}]$, set  
			$\Gamma_{j} = \ell_{\phi}(\cdot; \hat x_{k-1})$, and compute $k$-th cycle tolerance $\delta_{k}$ as
                  \begin{equation}
                           \delta_{k}  = \beta(\phi(\hat y_{k-1})-\phi_*)+\frac{(1-2\beta)\varepsilon}{2};\label{def:deltaj}
                  \end{equation}

	    \item[3.] compute the optimal solution $x_j$ and optimal value $m_j$ of \eqref{def:xj-1}, pair $(y_j,t_j)$ as

     \begin{equation}
      y_j := \argmin
     \left \{\phi(x) : x \in \{y_{j-1},x_j\} \right\}, \quad 
     t_j := \phi(y_j) - m_{j};\label{def:yj}
\end{equation}
		\item[4.] {\textbf{if}} $t_j \le \delta_k$ {\textbf{then}} go to step 5; \textbf{else}  perform a {\textbf{null update}, i.e.:}  
          
           \begin{itemize}
               \item [4.a]   {\bf (computing the next prox stepsize)} \textbf{if} either $j = j_{k-1}+1$ or
\begin{equation}\label{keyineq2}
        t_j -\tau t_{j-1} \le (1-\tau)\frac{\delta_k}{2}    ,
     \end{equation}
      \textbf{then} set $\lam_{j+1}=\lam_{j}$; \textbf{else}, set
$\lam_{j+1}=\lam_{j}/2$; 
    \item [4.b]   {\bf (computing the next bundle)}  let $\bar \Gamma_j$ be a shadow of $\Gamma_j$ for \eqref{def:xj-1} and set   \begin{equation}\label{def:Gamma}
\Gamma_{j+1}(\cdot) = \max \left\{\ell_{\phi}(\cdot;x_j), \bar{\Gamma}_j(\cdot)\right\};
\end{equation}
     \item [4.c] ({\bf updating the  iteration index)} set  $j\leftarrow j+1$ and go to step 3;
           \end{itemize}  
	\item [5.]
       perform a {\bf serious update}, i.e.:
          \begin{itemize}
              \item[5.a]
              {\bf (computing  the next prox center)}    set 
            $\hat x_k=x_j$;
            \item[5.b] ({\bf recording the final cycle output)}
            $j_{k} = j$, and $(\hat \lambda_k, \hat y_k, \hat \Gamma_k,\hat m_k,\hat t_k)=(\lambda_j, y_j, \Gamma_j,m_j,t_j)$;
             \item[5.c] ({\bf updating the cycle and iteration indices)} set $ k \leftarrow k+1$ and  $j\leftarrow j+1$, and go to step 1.
          \end{itemize}

           

      
        
	\end{itemize}
	\rule[0.5ex]{1\columnwidth}{1pt}
    
  We now introduce some terminology related to Ad-GPB*.  Ad-GPB* performs two types of iterations, namely,
null and serious, corresponding to the kinds of updates performed at the end.
 The index $j$ counts the iterations (including null and serious).
 Let $
   j_1 \le j_2 \le \ldots$ denote the sequence of all serious iterations
   (i.e., the ones ending with a serious update) and, for every $k \ge 1$,
   define
   $i_k=j_{k-1}+1$ and the $k$-th cycle ${\cal C}_k$
   as 
  \begin{equation}\label{def:Ck}
     {\cal C}_k := \{ i_k,\ldots,j_k\}.
 \end{equation} 
  Hence, index $k$ counts the
 cycles generated by Ad-GPB*.
 Observe that steps 2 and 5.b of Ad-GPB* imply that for every $k \ge 1$, 
\begin{equation}\label{obs:lamik}
      \lam_{i_k} \in [\hat \lam_{k-1}, 2 \hat \lam_{k-1}], \quad  \lam_{j_k} = \hat \lambda_k .
   \end{equation}
 An iteration $j \in  {\cal C}_k$ is called 
 good (resp., bad) if $\lam_{j+1} \ge \lam_j$ (resp., $\lam_{j+1} = \lam_j/2$). 
 Observe that   $i_k$ and $j_k$  are always   good iterations in view of \eqref{obs:lamik}, the fact that $i_k = j_{k-1}+1$ and the test in step 4.a of Ad-GPB*. Moreover,   \eqref{keyineq2} is violated whenever
 $j \in  {\cal C}_k\setminus\{j_k,i_k\}$ is a bad iteration.


We now make some important remarks about Ad-GPB*.
First,
   step 2 of Ad-GPB*, identity \eqref{def:Gamma}, and the definitions of $\bar \Gamma_j$  and  ${\cal B}(\phi)$ in  Definition \ref{def:shadow} and \S \ref{subsec:DefNot}, respectively, imply that 
  \begin{equation}\label{phi-property}
        \Gamma_j \le \phi, \quad
    \Gamma_j \in \bConv{n}, \quad \dom \Gamma_j = \dom h \quad \forall j \ge 1,
    \end{equation}
    and hence that
$\hat \Gamma_k \in \bConv{n}$ and $\hat \Gamma_k \le \phi$, due to the fact that
  $\hat \Gamma_k$ is the last $\Gamma_j$
  generated within the $k$-th cycle. 
Second, it follows from  \eqref{phi-property} and the definition of $x_j$ in \eqref{def:xj-1}   that  $x_j \in \dom h$ for every $j \ge 1$. Third, an induction argument using \eqref{def:yj}
and the fact that $y_0=\hat x_0 \in \dom h$ imply that 
$y_j \in \{\hat x_0, x_1, \ldots, x_j\} \subset \dom h$ and 
\begin{equation}\label{eq:minseq}
	 	y_j  \in \Argmin \left\lbrace \phi(x) :
		x \in \{ \hat x_{0}, x_1,\ldots,x_{j}\}	\right\rbrace
\end{equation} 
(hence, $\phi(y_{j+1}) \le \phi(y_{j})$) for every $j \ge 1$.

We now make some further remarks about Ad-GPB* in order to compare it with GPB of \cite{liang2024unified}. First, the formula for $\delta_k$ in \eqref{def:deltaj} with $\beta=0$ reduces to $\delta_k = \varepsilon/2$ and hence does not depend on $\phi_*$. Hence, as with GPB, Ad-GPB* with $\beta = 0$ only uses $\phi_*$ in its termination criterion in step 1. Second, while GPB keeps $\lam$ constant throughout the iterations, Ad-GPB* may adjust $\lam$ in two ways, namely: i) within the cycles in a delayed manner without performing any line search (see step 4.a of Ad-GPB*); ii) while choosing the initial prox stepsize for the $k$-th cycle, i.e., in the interval $[\hat \lam_{k-1}, 2 \hat \lam_{k-1}]$ where
$\hat \lam_{k-1}$ is the final prox stepsize from the previous cycle (see step 2 of Ad-GPB*).
Finally,  
if Ad-GPB* reaches step 2, then 
 the primal gap $\phi(\hat y_{k-1})-\phi_*$ exceeds
 $\varepsilon$ due to step 1, and can be substantially larger than $\varepsilon$ during its early cycles.
 Hence,
 the right-hand side $\delta_k$ in \eqref{def:deltaj}
 used on the cycle termination criterion in step 4 is always larger than  $\varepsilon/2$ and can be substantially larger than
$\varepsilon/2$
 during the early cycles. Since, in contrast, GPB terminates a cycle when the inequality $t_j \le \varepsilon/2$ is satisfied,  the cycle termination of Ad-GPB* is  looser, and potentially much looser during early cycles, than that of GPB. Thus, Ad-GPB*
allows earlier cycles to terminate in fewer inner iterations,
which results in better overall performance.



 

Lastly,  Ad-GPB* never restarts a cycle.  Instead, its $k$-th cycle uses a key stepsize update scheme that  generates, without backtracking,  a final $k$-th cycle stepsize $\hat \lam_k$ for which \eqref{subpro} is inexactly solved (see step 4.a of Ad-GPB*).



    
     
 

We now state the main complexity result for Ad-GPB*. 
\begin{theorem}\label{main1}
Define 
\begin{align}
\bar t&:=  4D(4M_f+L_fD), \qquad 
Q(\varepsilon):= \frac{8(1-\tau)}{\tau} \left( 16 M_f^2 + L_f \varepsilon\right)+ \frac{ \varepsilon}{\hat \lam_0  } , \label{def:bar t}
\end{align}
\begin{align}
 \hat K(\varepsilon) 
 &:=
   \left\lceil 8\left(d(\hat x_0)^2+\beta\hat \lam_0  [ \phi(\hat x_0) - \phi_*]\right)  \frac{Q(\varepsilon)}{\varepsilon^2}\right\rceil ,
\label{defhatK}\end{align}
where  the tuple $(\hat x_0,\varepsilon,\hat \lam_0,\tau)$ is the first four inputs of Ad-GPB*.
Then, the following statements about Ad-GPB* hold:
\begin{itemize}
    \item [a)] it terminates 
 in at most
$\hat K(\varepsilon)$ cycles;
\item [b)] it performs 
 in at most  \begin{equation}\label{def:totalknwon}
   \frac{3\hat K(\varepsilon)}{1-\tau} \left[
   \ln^+ \left(  \frac{\bar t}{\varepsilon} \right)
    +1
    \right]  +\log_2^+ \left( \frac{\hat \lam_0 Q(\varepsilon)}{\varepsilon}\right)
    \end{equation}
       total iterations and its output
         $\hat x$ satisfies \begin{equation} \label{Outputy}
             \phi(\hat x)  \le \min \{\phi(\hat x_0),\phi_*+\varepsilon\}.
             \end{equation}
\end{itemize}
%
\end{theorem} 

 We now discuss some implications of the complexity bound derived in Theorem \ref{main1}. 

 \begin{itemize}
 \item
 For the remarks in this item, assume that $\tau^{-1} = {\cal O}(1)$.
 In this case,
 the iteration complexity of Ad-GPB* to find an $\varepsilon$-solution of \eqref{eq:ProbIntro} is 
 \begin{equation}\label{simplifiedcombound}
 \tilde {\cal O}\left(\frac{ \left(d(\hat x_0)^2+\beta\hat \lam_0  [ \phi(\hat x_0) - \phi_*] \right)}{\varepsilon^2}\left(M_f^2+\varepsilon L_f+\frac{\varepsilon}{(1-\tau)\hat \lam_0 }\right)\right).
\end{equation}
Moreover, if an upper bound $D$ on $d(\hat x_0)$ is known and
$\hat \lam_0$ satisfies 
\[
\hat \lam_0={\Theta}\left(\frac{D^2}{\phi(\hat x_0)-\phi_*}\right),
\]
then the  complexity of Ad-GPB* becomes
\begin{equation}\label{def:comp1}
 \tilde {\cal O}\left(\frac{ D^2 }{\varepsilon^2}\left(M_f^2+\varepsilon L_f +\frac{(\phi(\hat x_0)-\phi_*)\varepsilon}{(1-\tau)D^2 }\right)\right).
\end{equation}
Finally, in terms of $\varepsilon$ only, if $(1-\tau)^{-1} =   {\cal O}(\varepsilon^{-1})$, then bound
\eqref{def:comp1} is $\tilde {\cal O}(\varepsilon^{-2})$.
\item
For the remarks in this item, assume that
$\max\{ \tau^{-1}, (1-\tau)^{-1} \} =  {\cal O}(1)$.
Another way of choosing $\hat \lam_0$ is 
a multiple of the Polyak stepsize, i.e.,
\[
\hat \lam_0= \frac{\alpha(\phi(\hat x_0)-\phi_*)}{\|\phi'(\hat x_0)\|^2}
\]
where  the multiple factor $\alpha$ is such that $0<\alpha = {\cal O}(1)$ and $\phi'(\hat x_0)$ is a subgradient of $\phi$ at $\hat x_0$. This choice of $\hat \lam_0$ and  \eqref{def:d0} then imply that 
\begin{align*}
\hat \lam_0 [\phi(\hat x_0)-\phi_*]  = \frac{\alpha[\phi(\hat x_0)-\phi_*]^2}{\|\phi'(\hat x_0)\|^2} \le \frac{\alpha \left(\inner{\phi'(\hat x_0)}{\hat x_0-x^*}\right)^2}{\|\phi'(\hat x_0)\|^2} \le \alpha\|\hat x_0-x^*\|^2\overset{\eqref{def:d0}}= \alpha d(\hat x_0)^2,
\end{align*}
where the first and second inequalities follow from the subgradient and the Cauchy--Schwarz inequalities, respectively.
Thus,  assuming without loss of generality that $\phi(\hat x_0)-\phi_* > \varepsilon$, the iteration complexity bound \eqref{simplifiedcombound}  becomes
 \begin{equation}\label{def:comp2}
 \tilde {\cal O}\left(\frac{ d(\hat x_0)^2 }{\varepsilon^2}\left(M_f^2+\varepsilon L_f +\|\phi'(\hat x_0)\|^2\right)\right).
 \end{equation}
 Finally, in terms of $\varepsilon$ only, the bound \eqref{def:comp2} is $\tilde {\cal O}(\varepsilon^{-2})$;
\item For the remarks in this item, assume that $\beta = 0$ and
$\max\{ \tau^{-1}, (1-\tau)^{-1} \} =  {\cal O}(1)$. Then, the complexity of Ad-GPB* is 
 \[
  \tilde {\cal O}\left(\frac{ d(\hat x_0)^2}{\varepsilon^2}\left(M_f^2+\varepsilon L_f +\frac{\varepsilon}{\hat \lam_0 }\right)\right),
 \]
 which reduces to
 \begin{equation}\label{eq:beta0comp}
  \tilde {\cal O}\left(\frac{ d(\hat x_0)^2}{\varepsilon^2}\left(M_f^2+\varepsilon L_f \right)\right)
 \end{equation}
 when the initial prox stepsize $\hat \lam_0$ satisfies $
\hat \lam_0^{-1} = {\cal O}\left( (M_f^2+\varepsilon L_f)/\varepsilon\right)$, and hence is
${\Omega}(\varepsilon)$.
 This contrasts with GPB of \cite{liang2024unified} whose complexity reduces to \eqref{eq:beta0comp} only when the (constant) prox stepsize $\lam$ satisfies
 \[
 \lam^{-1} = {\cal O}\left( \frac{M_f^2+\varepsilon L_f}{\varepsilon}\right), \quad \lam = {\cal O}\left(\frac{1}{\varepsilon}\right).
 \]
 and hence is not only $\Omega(\varepsilon)$ but also 
 ${\cal O}(\varepsilon^{-1})$.
 
\end{itemize}

Under the more general assumption that
\(\phi\) is \(\mu\)-convex for some \(\mu\geq 0\), we next describe the complexity of Ad-GPB* when its tolerance is chosen 
as $\delta(\varepsilon,\gamma,\hat x_0)
:=
\varepsilon +\gamma[\phi(\hat x_0)-\phi_*]$
for some \(\varepsilon>0\) and \(\gamma\in (0,1)\).
This hybrid tolerance consists of the sum of an absolute tolerance $\varepsilon$ and a relative tolerance $\gamma[\phi(\hat x_0)-\phi_*]$.

The following result, whose  proof is deferred to
Section~\ref{outer:Ad-GPB}, is not only interesting in its own
right, but it will also be used in the analysis of a restart version of Ad-GPB* presented in Section~\ref{sec:strongconvex}.

\begin{prop}
\label{keyleema:strongly}
Suppose that \(\phi\) is \(\mu\)-convex for some \(\mu\geq 0\) and
define 
\begin{equation}\label{def:delta}
    \delta=\delta(\varepsilon,\gamma,\hat x_0):=
\varepsilon +\gamma[\phi(\hat x_0)-\phi_*],
\end{equation}
\begin{equation}\label{def:Kmu}
\hat J_\mu(\varepsilon) = \hat J_\mu(\varepsilon,\gamma,\hat x_0)
:=
\left \lceil \frac{8Q(\varepsilon)}{\varepsilon}
\min\left\{
  \frac{1+\mu\beta\hat\lambda_0}{\gamma\mu},
    \;
    \frac{d(\hat x_0)^2+\beta\hat\lambda_0
    \bigl[\phi(\hat x_0)-\phi_*\bigr]}{\varepsilon}
\right\}\right \rceil,
\end{equation}
where \(Q(\cdot)\) is as in \eqref{def:bar t}.
Then,  Ad-GPB*$(\hat x_0,\delta
,\hat \lam_0,\tau,\beta)$
outputs a point $\hat x$ satisfying
\begin{equation}\label{ineq:subbddmu}
    \phi(\hat x)-\phi_* \le \delta =
    \varepsilon
    +
    \gamma[ \phi(\hat x_{0})-\phi_*]
\end{equation}
in a total number of iterations bounded by
\begin{equation}\label{Kmu}
         \frac{3 \hat J_{\mu}(\varepsilon)}{1-\tau} \left[
   \ln^+ \left(  \frac{\bar t}{\varepsilon} \right)
    +1
    \right]  +\log_2^+ \left( \frac{\hat \lam_0 Q(\varepsilon)}{\varepsilon}\right).
\end{equation}

 %
\end{prop}


We now make a few remarks about Proposition~\ref{keyleema:strongly}. First, when \(\mu=0\), the complexity bound \eqref{Kmu} reduces to the bound in \eqref{def:totalknwon}. Second, under the assumption that  \(\mu>0\) and \(2\varepsilon \le d(\hat x_0)^2\gamma\mu\), the second term within the minimum in \eqref{def:Kmu} is larger than the first one, because \(\phi(\hat x_0)-\phi_* \ge \mu d(\hat x_0)^2/2\). Hence, under this assumption, the complexity bound \eqref{Kmu} is \(\widetilde{\mathcal O}((\varepsilon\gamma\mu)^{-1})\) in terms of  \(\varepsilon\), \(\gamma\), and \(\mu\), only.

    \subsection{Special Instances of Shadow Functions} 
    \label{sec:shadow}   
    This subsection presents two concrete ways of implementing shadow functions.
    
    Noting that the shadow $\bar \Gamma_j$ in step 4.b of Ad-GPB* is undetermined, Ad-GPB* has the flexibility to choose $\bar \Gamma_j$. Next, we present two specific schemes of constructing $\bar \Gamma_j$, and hence two concrete ways to implement step 4.b of Ad-GPB*. 


	\begin{itemize}
	   
	    \item \textbf{2-cut:} Now we describe how  this scheme generates shadow $\{\bar \Gamma_j\}$ within $k$-th cycle. Every shadow generated by this scheme  is of the form $\bar \Gamma_j = A_{j}+h$ where $A_j$ is
        an affine  function  satisfying $A_j \le f$.  
        Note that the form of $\bar \Gamma_j$ above,  the definition of $\ell_{\phi}(\cdot;\cdot)$ in \eqref{def:ellphi}, and identity \eqref{def:Gamma}, imply that
        \[
        \Gamma_{j+1}(\cdot)=\max \left\{\ell_{\phi}(\cdot;x_j), \bar{\Gamma}_j(\cdot)\right\} = \max \left\{\tilde \ell_{f}(\cdot;x_{j}), A_{j}(\cdot)\right\}+h(\cdot),
        \]
        and hence that \eqref{def:xj-1} is equivalent to
 \begin{equation}\label{eq:xj-alt}
            x_{j+1} = \underset{u\in \R^n, t\in \R}\argmin \left \{t +h(u)+\frac{1}{2\lam_{j+1}}\|u-\hat x_{k-1}\|^2: \, \tilde \ell_{f}(u;x_{j}) - t \le 0, \,\, A_{j}(u) -t \le 0\right\}.
        \end{equation}
        The sequence $\{A_j\}$, and hence $\{\bar \Gamma_j:=A_j+h\}$, over the $k$-th cycle is recursively generated as  
\begin{equation}\label{def:Atk}
A_{j+1}(\cdot) = \left\{
\begin{array}{ll}
\tilde \ell_f(\cdot;\hat x_{k-1}),& \mbox{ if }j=j_{k-1},\\
  \theta_j A_{j}(\cdot) + (1-\theta_j) \tilde \ell_f(\cdot;x_{j}),& \mbox{ if }j \in {\cal C}_k \setminus \{j_k\} ,
    \end{array}
    \right.
    \end{equation} 
        where $\theta_j\in [0,1]$ for $j \in {\cal C}_k \setminus \{j_k\}$
is the Lagrangian multiplier for the second constraint 
$A_{j}(u) - t \le 0$ of \eqref{eq:xj-alt}. Using the optimality conditions for \eqref{eq:xj-alt}, it can be shown that $\bar \Gamma_j$ is a shadow of $\Gamma_j$ for \eqref{def:xj-1}
for every $j \ge 1$.

	     \item  \textbf{multi-cut:}     This scheme sets $\bar \Gamma_j(\cdot) = \Gamma_j(\cdot)$.  
         It is obvious that $\bar \Gamma_j$ is a shadow of $\Gamma_j$ for \eqref{def:xj-1}. 
         
	\end{itemize}

     We now make a remark about step 4.b of Ad-GPB*. Instead of following \eqref{def:Gamma}, \cite{liang2024unified} replaces \eqref{def:Gamma} by the weaker inequality $\Gamma^+ (\cdot) \ge  \tau \bar \Gamma(\cdot) +(1-\tau)  \ell_{\phi}(\cdot;x)$ for some $\tau \in (0,1)$ and, as
a result, contains the one-cut bundle update scheme described in \S 
3.1 of \cite{liang2024unified}.

\section{Proof of Theorem \ref{main1} and Proposition \ref{keyleema:strongly}}\label{sec:proof}
This section contains two subsections. The first one establishes  a bound on the  lengths of the cycles generated by Ad-GPB*. The second one establishes a bound on the number of cycles generated by Ad-GPB* and provides the proofs of Theorem \ref{main1} and Proposition \ref{keyleema:strongly}.

Recall that this section assumes that \eqref{keyass} holds.

\subsection{Bounding cycle lengths of Ad-GPB*}\label{cyclen}


Recall from \eqref{def:Ck} that $i_k$ (resp., $j_k$) denotes the first (resp., last) iteration index of the $k$-th cycle of Ad-GPB*.
The first result describes some basic facts about the iterations within any given cycle.
\begin{lemma}\label{lem:101}
	    For every $j\in {\cal C}_k\setminus \{i_k\}$, the following statements hold:
	    \begin{itemize}
	        \item[a)] there exists a function $\bar \Gamma_{j-1}(\cdot)$ such that
	        \begin{align}
	            &\max \left \{ \bar \Gamma_{j-1}(\cdot),\ell_{\phi}(\cdot;x_{j-1})\right\} = \Gamma_{j}(\cdot) \le \phi(\cdot), \label{eq:Gamma_j-1} \\
	            &\bar \Gamma_{j-1} \in \Conv{n}, \quad \bar \Gamma_{j-1}(x_{j-1}) = \Gamma_{j-1}(x_{j-1}), \label{property}\\
	        &x_{j-1} = \underset{u\in \R^n}\argmin \left \{\bar \Gamma_{j-1} (u) + \frac{1}{2\lam_{j-1}} \|u-\hat x_{k-1}\|^2 \right\}; \label{eq:relation}
	        \end{align} 
	        \item[b)]  for every $u\in \R^n$, we have
       \begin{equation}\label{ineq:Gammaj-1}
           \bar \Gamma_{j-1}(u) + \frac1{2\lam_{j-1}}\|u-\hat x_{k-1}\|^2\ge m_{j-1}+ \frac1{2 \lam_{j-1}}\|u-x_{j-1}\|^2 .
       \end{equation}
	    \end{itemize}
	\end{lemma}
	\begin{proof}
	    a) This statement immediately follows from \eqref{def:Gamma}, the definition of $\bar \Gamma_j$ in Definition \ref{def:shadow}, all with $j=j-1$. 
	    
	    b)
	    Using \eqref{eq:relation} and the fact that $ f=\bar \Gamma_{j-1} + \|\cdot-\hat x_{k-1}\|^2/(2\lam_{j-1}) $ is $  \lam_{j-1}^{-1} $ strongly convex, we have for every $u\in\dom h$,
	    \[
	    \bar \Gamma_{j-1}(u) + \frac{1}{2\lam_{j-1}} \|u-\hat x_{k-1}\|^2  \ge \bar \Gamma_{j-1}(x_{j-1}) + \frac{1}{2\lam_{j-1}} \|x_{j-1}-\hat x_{k-1}\|^2  + \frac{1}{2 \lam_{j-1}}\|u-x_{j-1}\|^2.
	    \]
	    The statement follows from the above inequality, the second identity in \eqref{property}, and the definition of $m_j$ in step 3 of Ad-GPB*. 
	\end{proof}
	

The next result presents some basic recursive inequalities for $\{t_j\}$.

\begin{lemma}\label{keyinner}
For every $j\in {\cal C}_k\setminus \{i_k\}$, the following statements hold:
\begin{itemize}
    \item [a)] for every $\tau' \in [0,1]$,  there holds 
    \[
    t_j -\tau' t_{j-1} \le 2 M_f (1-\tau') \|x_j-x_{j-1}\|  - \left(\frac{\tau'} {2\lam_{j-1}} - \frac{(1-\tau') L_f}{2}\right) \|x_j-x_{j-1}\|^2- \frac{1-\tau'} {2\lam_j} \|x_j-\hat x_{k-1}\|^2 ;
    \]
    \item[b)] if $\lam_{j-1} \le \tau/(2(1-\tau)L_f)$, then we have
    \begin{equation}\label{ineq:re_t2}
    t_j -\tau t_{j-1} \le  \frac{4M_f^2(1-\tau)^2\lam_{j-1}}{\tau }.
 \end{equation}
\end{itemize}
\end{lemma}

\begin{proof}
a) Inequality \eqref{eq:Gamma_j-1} implies that for every $\tau' \in [0,1]$, we have
\begin{align*}
\Gamma_j(x_j) &= \max \left \{ \bar \Gamma_{j-1}(x_j),\ell_{\phi}(x_j;x_{j-1}) \right\}  \ge (1-\tau') \ell_\phi (x_j;x_{j-1}) + \tau' \bar \Gamma_{j-1}(x_j).
\end{align*}
The definition of $m_j$ in step 3 of Ad-GPB*, the above inequality,   and \eqref{ineq:Gammaj-1} with $u=x_j$,  imply that
\begin{align*}
  m_j 
  &\ge (1-\tau') \ell_\phi (x_j;x_{j-1}) + \tau' \bar \Gamma_{j-1}(x_j) +   \frac1{2\lam_j} \|x_j-\hat x_{k-1}\|^2 \\
  &= (1-\tau') \left[ \ell_\phi (x_j;x_{j-1}) +\frac1{2\lam_j} \|x_j-\hat x_{k-1}\|^2 \right] +  \tau' \left[\bar \Gamma_{j-1}(x_j) +   \frac1{2\lam_j} \|x_j-\hat x_{k-1}\|^2 \right] \\
    &\overset{\lam_j \le \lam_{j-1}}\ge  (1-\tau') \left[ \ell_\phi (x_j;x_{j-1}) +\frac1{2\lam_j} \|x_j-\hat x_{k-1}\|^2 \right] +  \tau' \left[\bar \Gamma_{j-1}(x_j) +   \frac1{2\lam_{j-1}} \|x_j-\hat x_{k-1}\|^2 \right] \\
  &\overset{\eqref{ineq:Gammaj-1}}\ge (1-\tau') \left[ \ell_\phi (x_j;x_{j-1}) +\frac1{2\lam_j} \|x_j-\hat x_{k-1}\|^2 \right] + \tau' \left[ m_{j-1}+ \frac1{2\lam_{j-1}} \|x_j-x_{j-1}\|^2\right].
\end{align*}
Using this inequality and the definition of $t_j$ in \eqref{def:yj}, we have
\begin{align*}
    & t_j-\tau' t_{j-1} = [ \phi(y_j) - m_j ] - \tau' [ \phi(y_{j-1}) - m_{j-1}] \\
    &= [ \phi(y_j) - \tau'  \phi(y_{j-1}) ] - [ m_j - \tau' m_{j-1}] \\
    &\le  [ \phi(y_j) - \tau'  \phi(y_{j-1}) ] - (1-\tau') \left[ \ell_\phi (x_j;x_{j-1}) +\frac1{2\lam_j} \|x_j-\hat x_{k-1}\|^2 \right] -  \frac{\tau'}{2\lam_{j-1}} \|x_j-x_{j-1}\|^2 \\
    &=  [ \phi(y_j) - \tau'  \phi(y_{j-1}) - (1-\tau') \phi(x_j)] \\
    & \ \ \ +
    (1-\tau') \left[ \phi(x_j) - \ell_\phi (x_j;x_{j-1}) \right]  - \frac{1-\tau'} {2\lam_j} \|x_j-\hat x_{k-1}\|^2 -  \frac{\tau'}{2\lam_{j-1}} \|x_j-x_{j-1}\|^2\\
    &\le  (1-\tau') \left[ \phi(x_j) - \ell_\phi (x_j;x_{j-1}) \right]  - \frac{1-\tau'} {2\lam_j} \|x_j-\hat x_{k-1}\|^2 -  \frac{\tau'}{2\lam_{j-1}} \|x_j-x_{j-1}\|^2,
\end{align*}
where the last inequality is due to the definition of $y_j$ in \eqref{def:yj}. The
conclusion of the statement now follows from the above inequality and relation \eqref{ineq:est} with $(y,x)=(x_{j-1},x_j)$.

b) Using the assumption of this statement  and statement a)
with $\tau'=\tau$, we easily see that  
\[
t_j -\tau t_{j-1} \le 2 M_f (1-\tau) \|x_j-x_{j-1}\|  -  \frac{\tau  }{4\lam_{j-1}}  \|x_j-x_{j-1}\|^2.
\]
The statement now follows from the above inequality and the inequality $  2ab - b^2 \le a^2 $ with 
\[
a=\frac{2M_f(1-\tau)\sqrt{\lam_{j-1}}}{\sqrt{\tau}}, \quad b=\frac{\sqrt{\tau}\|x_j-x_{j-1}\|}{2\sqrt{\lam_{j-1}}}.
\]
\end{proof}

The next result describes some properties about the stepsizes $\lam_j$ within any given cycle.  
It uses the fact that if $j\in {\cal C}_k \setminus \{i_k,j_k\}$ is a bad iteration of Ad-GPB*, then \eqref{keyineq2} is violated (see step 4.a of Ad-GPB* and the first paragraph following Ad-GPB*).

 \begin{lemma}\label{keycoro}  
    For every cycle index \(k\ge1\), the following statements hold:
 \begin{itemize}
         \item [a)] for every index $j \in {\cal C}_k$, we have 
    \[
     \lam_{j} \ge \min \left\{  \frac{\varepsilon}{Q(\varepsilon)},\lam_{i_k}  \right\}
     \]
    where $Q(\varepsilon)$ is defined in \eqref{def:bar t};
     \item [b)]
    the number of bad iterations in ${\cal C}_k$ 
    is bounded by
    $
    \log_2^+ \left( \lam_{i_k}Q(\varepsilon)\varepsilon^{-1})\right). 
    $
     \end{itemize}
 \end{lemma}

 \begin{proof}
     a)  Assume  for contradiction that
     there exists $j \in {\cal C}_k$
     such that
\begin{equation}\label{vio}
\lam_j < \min \left\{ \frac{\varepsilon}{Q(\varepsilon)},\lam_{i_k}  \right\},
\end{equation}
and that $j$ is the smallest index in ${\cal C}_k$ satisfying the above inequality.
We claim that this assumption implies that
\begin{equation}\label{claim:direct}
\frac{\lam_{j-2}}4 \le \frac{\lam_{j-1}}2=  \lam_j .
\end{equation}
Before showing the claim, we argue that \eqref{claim:direct}
implies the conclusion of the lemma. Indeed, 
noting that \eqref{claim:direct}  and the definition of $Q(\varepsilon)$ in \eqref{def:bar t} imply that $\lam_{j-2} \le 4 \varepsilon Q(\varepsilon)^{-1} \le \tau/(2(1-\tau)L_f)$, it follows from \eqref{ineq:re_t2} with $j=j-1$ and the definition of $Q(\varepsilon)$ in \eqref{def:bar t}  that
\begin{align*}
   t_{j-1} -\tau t_{j-2} &\overset{\eqref{ineq:re_t2}}\le \frac{4(1-\tau)^2\lam_{j-2}M_f^2}{\tau } \le \frac{16(1-\tau)^2 \lam_j M_f^2}{\tau }\le \frac{16(1-\tau)^2\varepsilon M_f^2}{\tau Q(\varepsilon) } \\
   &\overset{\eqref{def:bar t}}
   \le (1-\tau)\frac{\varepsilon}{8}  \le  (1-\tau)\frac{\delta_k}{2}, 
\end{align*}
    where the last equality is due to the fact that $\delta_k \ge \varepsilon/4$ in view of the definition of $\delta_k$ in \eqref{def:deltaj} and the fact that $\beta \le 1/4$. This conclusion then implies that \eqref{keyineq2} holds for iteration $j-1$,
    and hence that $\lam_j=\lam_{j-1}$ due to the logic of step 4.a of Ad-GPB*.
    Since this contradicts \eqref{claim:direct}, statement (a) follows.

    We will now show the above claim, i.e., that the definition of $j$ implies \eqref{claim:direct}. Indeed,
    since the logic of step 4.a implies that $\lam_{i_k+1} = \lam_{i_k}$ and $j$ is the smallest index in ${\cal  C}_k$ satisfying \eqref{vio}, we conclude that
 $j \ge i_k+2$ and $\lam_j \ne  \lam_{j-1}$.
Using these conclusions and the fact that the logic of
    step 4.a of Ad-GPB* implies that
    either $\lam_i = \lam_{i-1}$ or
$\lam_i = \lam_{i-1}/2$ for every $i \in {\cal C}_k\setminus\{i_k\}$, we then conclude that
both the inequality and the identity in \eqref{claim:direct} hold.

 b) First observe that for any $j \in \{i_k,\ldots,j_k-1\}$, we have
 $\lam_{j+1} = \lam_j$ (resp.,  $\lam_{j+1} = \lam_j/2$) if $j$ is a  good (resp., bad) iteration. Using this observation, it then follows that 
$\lam_{i_k}/ \hat \lam_k= 2^{s
_k}$. This observation together with (a) then implies that statement (b) holds.
 \end{proof}


We now make some remarks about Lemma \ref{keycoro}. First,  $\varepsilon/Q(\varepsilon)$ describes how the lower threshold of $ \lambda_j$  on the adaptive stepsize implicitly captures the extreme cases of $(M_f,L_f)$. Indeed,  $\varepsilon/Q(\varepsilon) = \Omega(\varepsilon/M_f^2)$ in the nonsmooth case where $L_f=0$ while  $\varepsilon/Q(\varepsilon) = \Omega(1/L_f)$ in the smooth case where $M_f=0$. Thus, the adaptive stepsize automatically adjusts to the problem structure without requiring any knowledge of $(M_f,L_f)$.

Second, it follows from Lemma \ref{keycoro}(b) that the number of bad iterations within the $k$-th cycle
${\cal C}_k$ is finite. 
Proposition \ref{inner} below provides a bound on $|{\cal C}_k|$, and hence shows that
every cycle ${\cal C}_k$ terminates.
Before showing this result,
we state two technical results
which provide some key properties about the sequence $\{t_j \}$.
 \begin{lemma}\label{ineq:a_j}
    The following statements hold:
    \begin{itemize}
       \item [a)] if $ j \in {\cal C}_k \setminus \{i_k\}$,  then $t_j \le t_{j-1}$;
        \item [b)]  if $ j \in {\cal C}_k \setminus \{i_k\}$ is a good iteration that is not the last one in ${\cal C}_k$, then
        \[
        t_{j} - \frac{\varepsilon}{8} \le \frac{2\tau}{1+\tau} \left (t_{j-1} - \frac{\varepsilon}{8}\right).
        \]
     
    \end{itemize}
\end{lemma}

\begin{proof}
a) The statement immediately follows from Lemma \ref{keyinner}(a) with $\tau'=1$.

b) Assume that $ j \in {\cal C}_k \setminus \{i_k\}$ is a good iteration that is not the last one in ${\cal C}_k$.
This together with the logic of the Ad-GPB* imply that \eqref{keyineq2} is satisfied
and the cycle-stopping criterion is violated at iteration $j$, i.e.,
\begin{equation}\label{notjk}
t_j  >  \delta_k.
\end{equation}
These two observations and the fact that $\tau <1$ then imply that
\begin{align*}
 t_{j} - \frac{\varepsilon}{8}  
 &\overset{\eqref{keyineq2}}\le \tau t_{j-1} +(1-\tau) \frac{\delta_k}{2} - \frac{\varepsilon}{8}\\
&\overset{\eqref{notjk}} \le \tau t_{j-1} +(1-\tau) \frac{t_j}{2} - \frac{\varepsilon}{8}
\\
& \overset{\tau <1}\le \tau \left(t_{j-1}- \frac{\varepsilon}{8} \right)+ \frac{1-\tau}{2}\left (t_{j}-\frac{\varepsilon}{8}\right) ,
\end{align*}   
which can be easily seen to imply that statement b) holds.

\end{proof}


\begin{lemma}\label{lem:t1j}
For every cycle index $k \ge 1$ generated by Ad-GPB*, the following statements hold:
    \begin{itemize}
        \item [a)]  $ t_{i
    _k}\le \bar t/8$ where $\bar t$ is as in \eqref{def:bar t};
       \item[b)]
        if $j \in {\cal C}_k $  is not the last iteration of ${\cal C}_k$, then
        \begin{equation}\label{keytj}
       t_{j} - \frac{\varepsilon}{8} \le \left(\frac{2\tau}{1+\tau} \right) ^{j-i_k-s_k} \left (t_{i_k} - \frac{\varepsilon}{8}\right)
         \end{equation}
        where $s_k$ denotes the number of bad iterations within cycle $k$.
    \end{itemize}
\end{lemma}
\begin{proof}
    a)  Using the facts that $\phi=f+h$ and $ \Gamma_{i_k}(\cdot) = \ell_{\phi}(\cdot;\hat x_{k-1}) $ (see  step 2 of Ad-GPB*), and the definition of $t_j$, $m_j$ and $y_j$ in step 3 of Ad-GPB*, respectively,  we have
	    \begin{align}
	        t_{i_k} &\overset{\eqref{def:yj}}= \phi(y_{i_k}) - m_{i_k} \overset{\eqref{def:yj}}\le \phi(x_{i_k}) - m_{i_k} 
  = \phi(x_{i_k}) - \Gamma_{i_k}(x_{i_k}) - \frac{1}{2\lam_{i_k}}\|x_{i_k}-\hat x
_{k-1}\|^2 \nonumber\\
	        &= f(x_{i_k}) - \tilde \ell_f(x_{i_k};\hat x_{k-1})- \frac{1}{2\lam_{i_k}}\|x_{i_k}-\hat x
_{k-1}\|^2 \nonumber\\
	        &\overset{\eqref{ineq:est}}\le 2M_f \|x_{i_k}-
         \hat x_{k-1}\| + \frac{L_f}{2}\|x_{i_k}-\hat x
_{k-1}\|^2 \nonumber\\
&\le 2M_f D + \frac{L_f}{2}D^2.
	    \end{align}
        where  the last inequality is due to Assumption (A1), and the fact that both  $x_{i_k}$ and $\hat x
_{k-1}$ are in $\dom h$.
Statement a) now follows from the above inequality and  the definition of $\bar t$ in \eqref{def:bar t}.

b) If $j-i_k-s_k \le 0$, then  \eqref{keytj} obviously follows. Assume then that  $j-i_k-s_k > 0$. The fact that there are at least $j-i_k-s_k$ good iterations
in $\{i_{k}+1, \ldots, j\}$,
 and statements a) and b) of Lemma \ref{ineq:a_j},   imply that  
 \begin{equation}\label{recuraj}
 t_{j} -\frac{\varepsilon}{8}\le \left (t_{i_k} -\frac{\varepsilon}{8}\right)\left (\frac{2\tau}{1+\tau}\right )^{j-i_k-s_k} 
\end{equation}
and thus the statement  follows from the above inequality. 
\end{proof}
\begin{prop}\label{inner}
For every cycle index $k \ge 1$ generated by Ad-GPB*,  the size of the $k$-th cycle is bounded by
    $|{\cal C}_k| \le s_k + \bar N(\varepsilon)+1 $ where
    $s_k$ denotes the number of bad iterations within it and $\bar N(\cdot)$ is defined as
    \begin{equation}\label{defNs}
    \bar N(\varepsilon) :=\left \lceil
     \frac{2}{1-\tau}\ln^+ \left( \frac{  \bar t}{   \varepsilon}  \right)  
    \right \rceil.
 \end{equation}
\end{prop}

\begin{proof} It suffices to prove that 
\[
|{\cal C}_k| \le s_k + \bar N_k(\varepsilon)+1
\]
where \begin{equation}\label{eq:realbarNk}
\bar N_k(\varepsilon):=\left \lceil
     \frac{\tau+1}{1-\tau}\ln^+ \left( \frac{   t_{i_k}}{   \varepsilon}  \right)  
    \right \rceil.
\end{equation}
Indeed, if the above claim is true, then the statement obviously follows from Lemma \ref{lem:t1j} (a), the definition of $\bar N (\varepsilon)$ in \eqref{defNs}, and the fact that $\tau \le 1$. Now we start to prove the above claim.
If $t_{i_k} <\varepsilon/8$, then the cycle-stopping criterion is satisfied with $j=i_k$. This implies that $|{\cal C}_k|=1$, and hence that the result trivially holds in this case. From now on, assume $t_{i_k} \ge \varepsilon/8$ and suppose for contradiction that $|{\cal C}_k|> s_k +\bar N_k(\varepsilon)+1$.
This implies that
there exists a nonnegative integer $J \ge i_k$ such that  $J+1 \in {\cal C}_k$
 and
 \begin{equation}\label{vio2}
 J-i_k+2 > s_k + \bar N_k(\varepsilon)+1
 \end{equation}
 because the left-hand side of \eqref{vio2} is the cardinality of the index set $\{i_k,\ldots, J+1\}$. 
Since   $J$ is not the last iteration of ${\cal C}_k$, the cycle-stopping criterion in step 4 of Ad-GPB* is violated with $j=J$, i.e.,
 \[
t_{J} >\delta_k  \ge \frac{\varepsilon}{4}
 \]
 where the last inequality is due to the definition of $\delta_k$ in \eqref{def:deltaj} and the fact that $\beta \le 1/4$.
This observation together with Lemma~\ref{lem:t1j}(b) with $j =J$ and statement a), then imply that
 \[
 \frac{\varepsilon}{8} \le t_{J} - \frac{\varepsilon}{8} \le \left(\frac{2\tau}{1+\tau} \right) ^{J-i_k-s_k} \left (t_{i_k} - \frac{\varepsilon}{8}\right)  \le \left(\frac{2\tau}{1+\tau} \right) ^{J-i_k-s_k}  t_{i_k} ,
\]
which together with the definition of $\bar N_k (\varepsilon)$ in \eqref{eq:realbarNk}  can be easily seen to imply that 
\[
J-i_k -s_k  \le  \bar N_k(\varepsilon) -1.
\]
 Since this conclusion contradicts \eqref{vio2},
the claim is proved.
\end{proof}

Before ending \S \ref{cyclen},
we now make an important remark about it. Even though the quantity $\delta_k$ defined in \eqref{def:deltaj} appears in the proofs of Lemmas \ref{keycoro}, \ref{ineq:a_j}, \ref{lem:t1j}, and Proposition \ref{inner}, the only fact that is used about it 
is that $\delta_k \ge \varepsilon/4$.  Thus, all the results in the analysis of \S \ref{cyclen} also hold under the more general assumption that $\delta_k \ge \varepsilon/4$ for every $k \ge 1$.





\subsection{Bounding number of cycles of Ad-GPB*}\label{outer:Ad-GPB}

 This subsection  establishes a bound on the number of cycles generated by Ad-GPB*. Even though the expression of $\delta_k$ in \eqref{def:deltaj} was not used in its full extent in the \S \ref{cyclen}, it will be so in this one.



We start by stating two technical results.
The first result describes some 
basic facts about the sextuple
$(\hat \lambda_k,\hat x_k, \hat y_k, \hat \Gamma_k,\hat m_k,\hat t_k)$
 obtained in step 5 of Ad-GPB*, and hence at the last iteration of its $k$-th cycle.

\begin{lemma}\label{lem:iterate}
		 For every cycle index $k$ of Ad-GPB*, the following statements hold:
		\begin{itemize}
		    \item[a)]
       $
      \hat \Gamma_k \in \bConv{n}$, $ \hat \Gamma_k \le \phi$, and $\dom \hat \Gamma_k = \dom h$;
   \item [b)] $\phi(\hat y_k) \le \phi(\hat y_{k-1})  \le \phi(\hat y_0)$; 
   \item[c)] for  every given $ u\in \dom h $, we have
   \begin{equation}\label{ineq:recu}
\phi(\hat y_k) - \hat \Gamma_k(u) \le
\hat t_k + \frac1{2\hat \lam_k} \left[ \|u-\hat x_{k-1}\|^2 - \|u-\hat x_{k}\|^2 \right].
\end{equation}

   \end{itemize}
\end{lemma}
\begin{proof}
    a)  This statement follows from 
\eqref{phi-property} and the fact that
  $\hat \Gamma_k$ is the last $\Gamma_j$
  generated within the $k$-th cycle.	
  
		
  

b) The  statement  follows from the fact that
  $\hat y_k$ is the last $y_j$ 
  generated within the $k$-th cycle and \eqref{eq:minseq} with $j=j_k$. 

   c) Observe that \eqref{phi-property} and the definitions of  the quantities  
  $\hat x_k$, $\hat m_k$, $\hat \Gamma_k$, and $\hat \lam_k$ in step 5 of Ad-GPB*, imply that  $(\hat x_k,\hat m_k)$ is the pair of optimal solution and optimal value of	\begin{equation}\label{optprob}		
			\min \left\lbrace \hat \Gamma_k(u) + \frac{1}{2\hat \lam_k}\|u- \hat x_{k-1} \|^2: u\in\R^n \right\rbrace.
			\end{equation}
   The above observation, the fact that $ \hat \Gamma_k(\cdot) + 1/(2\hat \lam_k)\|\cdot- \hat x_{k-1} \|^2$ is $  (1/\hat \lam_k)$-strongly convex, together imply that, 
		 for the given $ u\in \dom h $, we have
\begin{equation}
		    \hat m_k + \frac{1}{2\hat \lam_k} \|u-\hat x_{k}\|^2
		\le \hat \Gamma_k(u) + \frac1{2\hat \lam_k}\|u-\hat x_{k-1}\|^2,
		\end{equation}
	and hence that 
\begin{align*}
         \phi(\hat y_k) - \hat \Gamma_k(u) +\frac{1}{2\hat\lam_k}\| u -\hat x_{k} \|^2
			&\le \phi(\hat y_k) - \hat m_k + \frac1{2\hat \lam_k}\|u-\hat x_{k-1}\|^2
			=\hat t_k  + \frac{1}{2\hat \lam_k}\|u-\hat x_{k-1}\|^2, 
		\end{align*}
  where the equality is due to the definition of $\hat t_k$ in step 5.b of Ad-GPB*.
This shows that \eqref{ineq:recu}, and hence statement (c), holds.

\end{proof}

  The next result provides a uniform   lower bound
    on the sequence  $\hat \lam_k$ and also a bound on the total number of bad iterations generated by Ad-GPB*.

 \begin{lemma}\label{lem:t1} 
    The following statements hold for every cycle index $k\ge 1$ generated by Ad-GPB*:
    \begin{itemize}
        \item[a)]  $\hat \lam_k \ge \varepsilon/Q(\varepsilon)$ where $Q(\varepsilon)$ is as in \eqref{def:bar t}; as a consequence, there holds:
\begin{equation}\label{ineq:ublaminv}
     \frac{1}{\sum_{l=1}^k \hat \lam_l}  \le \frac{Q(\varepsilon)}{\varepsilon k} ; 
 \end{equation}
\item[b)] $\sum_{l=1}^k s_l \le\log_2^+(\hat \lam_0 Q(\varepsilon)\varepsilon^{-1})+k$ where
    $s_l$ denotes the number of bad iterations within cycle $l$.
    \end{itemize}
	\end{lemma}	
	\begin{proof}
a) Using Lemma \ref{keycoro}(a) with $j=j
_k$, we conclude that 
\[
     \hat \lam_{k}  \overset{\eqref{obs:lamik} }= \lam_{j_k} \overset{L. \ref{keycoro}(a)}\ge 
     \min \left\{ \frac{\varepsilon}{Q(\varepsilon)} ,  \lam_{i_k}  \right\}   \overset{\eqref{obs:lamik} }\ge \min \left\{ \frac{\varepsilon}{Q(\varepsilon)} , \hat \lam_{k-1}  \right\} ,\]
     where the equality and the second inequality follow from \eqref{obs:lamik}.
     Using the above inequality recursively and the definition of $Q(\varepsilon)$ in \eqref{def:bar t}, we then conclude that
     \[
       \hat \lam_{k} \ge \min \left\{ \frac{\varepsilon}{Q(\varepsilon)} , \hat \lam_{0}  \right\} \overset{\eqref{def:bar t}} = \frac{\varepsilon}{Q(\varepsilon)}.
     \]
 The above inequality then implies that 
\begin{align*}\label{obs:small}
\frac{1}{\sum_{l=1}^k \hat \lam_l} \le \frac{1}{k \min_{l=1,\ldots,k}\hat \lam_l} \le \frac{Q(\varepsilon)}{k \varepsilon},
\end{align*}
and hence that a) holds.

b) Recalling that $s_l$ denotes the number of bad iterations in the $l$-th cycle. Observe that for any $j \in \{i_l,\ldots,j_l-1\}$,  we have $\lam_{j+1}/\lam_j$ is equal to $1/2$ (resp., equal to $1$) if $j$ is a bad (resp., good) iteration. Using this observation, we easily see that 
$\hat \lam_l/ \lam_{i_l} = (1/2)^{s_l}$. Using this inequality and the fact that $\lam_{i_l} \le 2 \hat \lam_{l-1}$ in view of \eqref{obs:lamik}, we can conclude $\log_2 \hat \lam_{l-1} - \log_2 \hat \lam_l +1    \ge s_l$.
The statement  now follows 
by summing the last inequality from $l=1$ to $k$ and using  a).
	\end{proof}


We now make two remarks about Lemma \ref{lem:t1}. First, it follows from Proposition \ref{inner} and Lemma ~\ref{lem:t1}(b)  that the average number of inner iterations per outer performed by Ad-GPB* is
$\tilde {\cal O}(1)$.
Second, 
instead of the more relaxed condition \eqref{obs:lamik},  $\lam_{i_k}$ is chosen as $\lambda_{i_k}=\hat{\lambda}_{k-1}$, then it can actually be proved that the total (instead of the average) number of bad iterations is $\tilde {\cal O}(1)$.

The next lemma establishes an important primal gap bound that is used in the proof of 
Theorem~\ref{main1}. To state the result more compactly, for the initial point \(\hat x_0\), we use the shorthand
$x^*:=x^*(\hat x_0)$.

\begin{lemma}\label{lem:outer:rec}
For every cycle index $k\ge 1$ generated by Ad-GPB*, we have
 %
 \begin{equation}\label{outer:rec}
\sum_{l=1}^{k} \hat \lam_l [ \phi(\hat y_l) - \phi_*]  \le  \frac{\varepsilon}{2} \left( \sum_{l=1}^{k} \hat \lam_l \right) +4\beta \hat \lam_0  [ \phi(\hat x_{0}) - \phi_*]  +\|\hat x_{0}-x^*\|^2 .    
 \end{equation}

\end{lemma}
\begin{proof} Let $1 \le l \le k$ be fixed.
  It  follows from the fact that the cycle-stopping criterion in the first line of step 4 of Ad-GPB* is satisfied at iteration $j_l$, the definition of $\delta_{l}$ in \eqref{def:deltaj}, and the definitions of the quantities  
   $\hat t_l$ and $\hat y_l$ in step 5.b of Ad-GPB* that 
   \begin{equation}\label{outer:*}
     \hat t_l  \le \beta(\phi(\hat y_{l-1}) - \phi_*)+\frac{(1-2\beta)\varepsilon}{2}.
   \end{equation}
  The above inequality,  \eqref{ineq:recu} with $u=x^*$, the definition of $x^*$ in \eqref{def:d0}, and the facts that $\hat \Gamma_l \le \phi$ and $\phi_*=\phi(x^*)$, imply that
  \begin{align*}
   \hat \lam_l [ \phi(\hat y_l) - \phi_*] 
 &\le \hat \lam_l [ \phi(\hat y_l) - \hat\Gamma_l(x^*)] \\&
 \le  \hat \lam_l \hat t_l  + \frac{1}{2}\|\hat x_{l-1}-x^*\|^2 -\frac{1}{2}\|\hat x_{l} - x^*\|^2\\
 &\overset{\eqref{outer:*}}\le \hat \lam_l \left[\beta( \phi(\hat y_{l-1}) - \phi_*)+\frac{(1-2\beta)\varepsilon}{2}\right]+ \frac{1}{2}\| \hat x_{l-1}- x^*\|^2 -\frac{1}{2}\|\hat x_{l} - x^*\|^2 \\
 &\le  \frac{(1-2\beta)\varepsilon}{2} \hat \lam_l+ 2 \beta \hat \lam_{l-1} (\phi(\hat y_{l-1}) - \phi_* ) + \frac{1}{2}\|\hat x_{l-1}-x^*\|^2 -\frac{1}{2}\|\hat x_{l} - x^*\|^2
    \end{align*}
    where the last inequality is due to the fact that $\hat \lam_l \le 2 \hat \lam_{l-1}$  in view of \eqref{obs:lamik}.
  Summing the above inequality from $l=1,\ldots,k$, we conclude that
  \[
   (1-2\beta)\sum_{l=1}^{k} \hat \lam_l [ \phi(\hat y_l) - \phi_*]  \le  \frac{(1-2\beta)\varepsilon}{2}  \left(\sum_{l=1}^{k} \hat \lam_l \right)+2\beta \hat \lam_0[ \phi(\hat y_{0}) - \phi_*] +  \frac{1}{2}\|\hat x_{0}-x^*\|^2 
  \]
  The  statement now follows from the above inequality and the facts that $\hat y_0 = \hat x_0$ and $1-2\beta \ge 1/2$ because $\beta \le 1/4$
  by assumption (see the input of Ad-GPB*).
\end{proof}

We are now ready to prove Theorem \ref{main1}.

\noindent
{\bf Proof of Theorem \ref{main1}:}
a) 
   Suppose for contradiction that  Ad-GPB* does not terminate at
   the end of the $\hat K(\varepsilon)$-th cycle. This implies that 
   $\phi(\hat y_k)-\phi_* > \varepsilon$ for every $k=1,\ldots,\hat K(\varepsilon)$.  Using  the definition of $\hat K(\varepsilon)$,  inequality \eqref{outer:rec} with  $k = \hat K(\varepsilon)$, and \eqref{ineq:ublaminv}, we conclude that

        \begin{align*}
      \frac{d(\hat x_0)^2+4\beta \hat \lam_0  [ \phi(\hat x_0) - \phi_*]}{\varepsilon} &\overset{\eqref{outer:rec}}\ge   \frac{1}{\varepsilon} \sum_{k=1}^{ K} \hat \lam_k [ \phi(\hat y_k) - \phi_*]  - \frac12 \sum_{k=1}^{K} \hat \lam_k> \sum_{k=1}^{ K} \hat \lam_k -\frac12 \sum_{k=1}^{K} \hat \lam_k \\
      &\overset{\eqref{ineq:ublaminv}}\ge  \frac{\varepsilon K}{2Q(\varepsilon)}\ge\frac{d(\hat x_0)^2+4\beta \hat \lam_0  [ \phi(\hat x_0) - \phi_*]}{\varepsilon} 
    \end{align*}
    which yields the desired contradiction.
    We have thus proved the above claim.
    
    
   b) Let $\bar k \le \hat K(\varepsilon)$ denote the number of cycles generated by Ad-GPB*.
     Proposition  \ref{inner} and  Lemma \ref{lem:t1}(b) then imply that the total number of iterations performed by  Ad-GPB*  is bounded by
    \begin{align*}
    \sum_{k=1}^{\bar k} &|{\cal C}_k| \le \sum_{k=1}^{\bar k} \left(
    \frac{2}{1-\tau}\ln^+ \left(  \frac{\bar t}{\varepsilon} \right) 
    +2 +s_k
    \right) \\
 &\le   \left(
    \frac{2}{1-\tau}\ln^+ \left(  \frac{\bar t}{\varepsilon} \right) 
    +2
    \right)\hat K(\varepsilon)+\hat K(\varepsilon) + \log_2^+  \frac{\hat \lam_0Q(\varepsilon)}{\varepsilon}\\
    & \le \frac{3\hat K(\varepsilon)}{1-\tau}\left(
    \ln^+ \left(  \frac{\bar t}{\varepsilon} \right) 
    +1
    \right)+ \log_2^+  \frac{\hat \lam_0 Q(\varepsilon)}{\varepsilon}
 \end{align*}
where the last inequality is due to  the fact that $1 \le 1/(1-\tau)$.     Finally, it outputs an iterate $\hat x$ satisfying \eqref{Outputy} due to the termination criterion in step 1 of Ad-GPB*, the second inequality in Lemma \ref{lem:iterate}(b) and the fact that $\hat y_0 = \hat x_0$. 
\endproof

We next prove  Proposition \ref{keyleema:strongly}.




\noindent
{\bf Proof of Proposition \ref{keyleema:strongly}:}
Recalling the definitions of $\delta$ and $\hat J_\mu(\cdot)$ in \eqref{def:delta} and \eqref{def:Kmu}, respectively, we claim that Ad-GPB*$(\hat x_0,\delta
,\hat \lam_0,\tau,\beta)$ terminates,
and hence obtains a point $\hat x$ satisfying \eqref{ineq:subbddmu}, in at most $\hat J_\mu(\varepsilon)$ cycles.
The conclusion of the proposition then follows
using the above claim, the facts that $\delta \ge \varepsilon$ and $Q(t)/t$ is a non-increasing function
of $t$, and  arguments similar to those used in the proof of   Theorem~\ref{main1}(b).

To show the above claim, it suffices to prove that Ad-GPB*$(\hat x_0,\delta
,\hat \lam_0,\tau,\beta)$ terminates in at most $
\hat J_\mu(\delta) $ cycles due to the facts that  
$\varepsilon \le \delta$ and $\hat J_\mu(\cdot)$   is  non-increasing.
Indeed, first note that Theorem \ref{main1}(a) with $\varepsilon = \delta$ implies that   Ad-GPB*$(\hat x_0,\delta
,\hat \lam_0,\tau,\beta)$ terminates in at most $
\hat K\left(\delta \right) $ cycles.
Since by \eqref{defhatK},   $\hat J_\mu(\delta) =  \min \{\hat K(\delta), \hat K_\mu\}$ where \begin{equation}\label{def:Kmuconvex}
 \hat K_\mu =  \hat K_\mu(\delta) :=\left \lceil\frac{8 Q(\delta)}{\delta}\left(   \frac{1 + \mu \beta \hat \lam_0 }{\gamma\mu}\right)   \right \rceil,
   \end{equation}
it suffices to show that, if $\mu>0$, Ad-GPB*$(\hat x_0,\delta
,\hat \lam_0,\tau,\beta)$  terminates in at most $\hat K_\mu$ cycles.
 To show this, suppose for contradiction that  Ad-GPB*$(\hat x_0,\delta
,\hat \lam_0,\tau,\beta)$ does not terminate at
   the end of the $(\hat K_{\mu})$-th cycle, and hence that 
   $\phi(\hat y_{\hat K_\mu})-\phi_* > \delta$. Lemma \ref{lem:iterate}(b) and \eqref{outer:rec} with $\varepsilon = \delta$ and $k = \hat K_\mu$, and the fact that $\phi$ is $\mu$-strongly convex, imply that 
 \begin{align*}
 [ \phi(\hat y_{\hat K_{\mu}}) - \phi_*]  \sum_{l=1}^{\hat K_{\mu}} \hat \lam_l &\overset{L.\ref{lem:iterate}(b)}\le \sum_{l=1}^{\hat K_{\mu}} \hat \lam_l [ \phi(\hat y_l) - \phi_*] \overset{\eqref{outer:rec}}\le    \frac{\delta}{2} \left(\sum_{l=1}^{\hat K_{\mu}} \hat \lam_l\right)+ 4 \beta\hat \lam_0  [ \phi(\hat x_{0}) - \phi_*]   + \|\hat x_{0}-x^*\|^2 \\ &\le  \frac{\delta}{2} \left( \sum_{l=1}^{\hat K_{\mu}} \hat \lam_l \right)+2 \left(  2\beta \hat \lam_0 +  \frac{1}{\mu } \right) \left [ \phi(\hat x_{0}) - \phi_*\right] . 
 \end{align*}
Dividing both sides of the above inequality  by $\sum_{l=1}^{\hat K_{\mu}} \hat \lam_l$ and using \eqref{ineq:ublaminv} with $(\varepsilon,k) =(\delta,\hat K_\mu)$, and the definition of $\delta$ in \eqref{def:delta}, we have
 \begin{align*}
 \phi(\hat y_{\hat K_{\mu}}) - \phi_* &\overset{\eqref{ineq:ublaminv}} \le \frac{\delta}{2}+ \frac{ 2(2\beta \mu\hat \lam_0 +  1 )Q(\delta)}{   \mu \delta \hat K_{\mu}}  \left[\phi(\hat x_{0}) - \phi_*\right] \\ 
 &\overset{\eqref{def:Kmuconvex}}
 \le \frac{\delta}2 + \frac{\gamma}2[\phi(\hat x_{0})-\phi_*] \overset{\eqref{def:delta}} \le \delta,  
\end{align*}
where   the second  inequality is due to definition of $\hat K_\mu$ in  \eqref{def:Kmuconvex}. Since the above inequality contradicts the previous conclusion that   $\phi(\hat y_{\hat K_\mu})-\phi_* > \delta$,  the above claim follows.
In view of the remark made in the first paragraph of the proof, the whole proposition
also follows. 
\endproof

 Even though (A1) in \S \ref{Sec:Ass} assumes
that $\dom h$ is bounded, this assumption is only used in the proof of Lemma~\ref{lem:t1j}(a) to show that the quantity $t_{i_k}$ is   bounded.
 A more complex proof that $t_{i_k}$ is  bounded without assuming that $\dom h$ is bounded
can be given by using similar arguments to those used in the proof of Lemma 4.7 of \cite{liang2024unified}.
Consequently, following the latter approach, the complexity results derived for Ad-GPB* in Theorem \ref{main1} and Proposition \ref{keyleema:strongly} can be extended to the setting where $\dom h$ is unbounded.

\section{General Case where $\phi_*$ is not known}\label{Sec:General}
     
    This section presents a variant of Ad-GPB*, referred as Ad-GPB, for the case where $\phi_*$ is not known. 
    
    We start by formally describing Ad-GPB.

{\small
\noindent\rule[0.5ex]{1\columnwidth}{1pt}
	Ad-GPB 
	
	\noindent\rule[0.5ex]{1\columnwidth}{1pt}
       \textbf{Input:} $(\hat x_0,\varepsilon) \in \dom h \times \R_{++}$,   $(\hat \lam_0,\tau) \in \R_{++} \times (0,1)$
	\begin{itemize}
		\item [0.] 
   set
	$\hat y_0=y_0= \hat x_0$, $j_0=0$, $\beta_1 = 1/4$, $j=k=1$, and
    \begin{equation}\label{def:hatell0}
       \hat \ell_0 = \inf_{u \in \dom h} \ell_{\phi}(u; \hat x_{0});
    \end{equation}
 \item [1.] \textbf{(stopping criterion)}  {\textbf{if}}
    $
    \phi(\hat y_{k-1})- \hat \ell_{k-1} \le \varepsilon$, 
 \textbf{then}
  output  $(\hat y,\hat \ell)=(\hat y_{k-1},\hat \ell_{k-1}) $, and \textbf{stop}; 
 \item [2.] {\bf (initializing the $k$-th cycle)}   choose $\lam_j = \hat \lam_{k-1}$,   set
			$\Gamma_{j} = \ell_{\phi}(\cdot; \hat x_{k-1})$    
   and compute the $k$-th cycle tolerance $\delta_{k}$ as 
        \begin{equation}
                 \delta_{k} = \beta_{k}[\phi(\hat y_{k-1})-\hat \ell_{k-1}]+\frac{\varepsilon}{4}\label{def:deltakj};
        \end{equation}
	    \item[3.] 
	    compute  the optimal solution $x_j$ and optimal value $m_j$ of \eqref{def:xj-1} and pair $(y_j,t_j)$ as in \eqref{def:yj};
		\item[4.] {\textbf{if}} $t_j \le \delta_k$ {\textbf{then}} go to step 5; \textbf{else}  perform a {\textbf{null update}, i.e.:}  
         
             \begin{itemize}
                 \item [4.a] {\bf (computing the next prox stepsize)}  {\textbf{if}} either $j = j_{k-1}+1$ or
         \begin{equation}\label{keyineq}
        t_j -\tau t_{j-1} \le (1-\tau)\frac{\delta_k}{2},
     \end{equation}
     {\textbf{then}} set $\lam_{j+1}=\lam_{j}$; {\textbf{else}}, set
$\lam_{j+1}=\lam_{j}/2$;
     \item [4.b] {\bf (computing the next  bundle)}  let $\bar \Gamma_j$ be a shadow of $\Gamma_j$ for \eqref{def:xj-1} and set   \begin{equation}\label{def:Gamma*}
\Gamma_{j+1}(\cdot) = \max \left\{\ell_{\phi}(\cdot;x_j), \bar{\Gamma}_j(\cdot)\right\};
\end{equation}
     \item [4.c] ({\bf updating the  iteration index}) set $j\leftarrow j+1$ and go to step 3;
           
         \end{itemize}
        \item [5.] perform a {\bf serious update}, i.e.:
          \begin{itemize}
            \item[5.a]
              {\bf (computing  the next prox center)}    set 
            $\hat x_k=x_j$;
            \item[5.b] ({\bf recording the final cycle output)}
            $j_{k} = j$, and $(\hat \lambda_k, \hat y_k, \hat \Gamma_k,\hat m_k,\hat t_k)=(\lambda_j, y_j, \Gamma_j,m_j,t_j)$;
          \item [5.c] 
          ({\bf updating $\hat \ell_k$}) compute
          \begin{align}
             \hat \ell_{k} &:= \max \left\{\hat \ell_{k-1},\inf_{u \in \dom h}  {\cal A}_k^a (u) \right\} \label{def:nk} 
 \end{align}
 where 
 \begin{equation}\label{def:hatgamma}
      {\cal A}^a_k (u) := \frac{ \sum_{l=\lceil k/2 \rceil}^k\hat \lam_l {\cal A}_l(u) }
 {\sum_{l=\lceil k/2 \rceil}^k \hat \lam_l}, \quad {\cal A}_k(u) := \hat \Gamma_k(\hat x_k)+\frac{1}{\hat \lam_k}\inner{\hat x_{k-1}-\hat x_k}{u-\hat x_k };
 \end{equation}
  
 \item [5.d] {\bf (computing  $\beta_{k+1}$) }
 compute
 \begin{align}
           \hat \Lambda_k &:= \sum_{l=\lceil k/2 \rceil}^k \hat \lam_l, \quad \hat \Phi_k : = \frac{\sum_{l=\lceil k/2 \rceil}^k \hat \lam_l\phi(\hat y_l)}{\hat \Lambda_k} 
           ,\quad 
           \hat g_k := \frac{\sum_{l =\lceil k/2 \rceil}^k \beta_{l} \hat \lam_l  \left[ \hat \ell_k - \hat \ell_{l-1} \right]}{\hat \Lambda_k};\label{def:gk}
           \color{black}
          \end{align}
            {\textbf{if}}
          $\hat g_k \le 
          (\hat \Phi_k -\hat \ell_k)/8$,  {\textbf{then}}
          set 
     $\beta_{k+1}=\beta_{k}$; {\textbf{else}}, set  $\beta_{k+1}=\beta_{k}/2$;
     \item [5.e]  ({\bf updating the cycle and iteration indices})  set $k \leftarrow k+1$ and   $j\leftarrow j+1$, and go to step 1.  
     \end{itemize}

	\end{itemize} 
	\rule[0.5ex]{1\columnwidth}{1pt}}

   We now make some remarks about Ad-GPB.
  
  First, the definition of ${\cal A}_k^a$ in 
  \eqref{def:hatgamma} and relation \eqref{phi-property} imply that ${\cal A}_k^a \in \bConv{n}$ and ${\cal A}_k^a \le \phi$,
  and hence 
   that 
  $ \inf_{u} {\cal A}_k^a(u) \le \inf_u \phi(u) = \phi_*$. 
  
  Second,  the facts that $\dom {\cal A}_k^a=\dom  h$ is bounded (see Assumption (A1))
  and 
  ${\cal A}_k^a$ is a closed convex function  imply that $\inf_u {\cal A}_k^a(u) > -\infty$.  
  Moreover, the problem
 $\inf_{u} {\cal A}_k^a(u)$ has the same format as the first one that appears in \eqref{easyso},
 and hence is easily solvable by assumption.
 Its optimal value, which is a lower bound on $\phi_*$ as already observed above, is used to
  update the lower bound $\hat \ell_{k-1}$ for $\phi_*$ to a new one that is no worse, namely, $\hat \ell_k \ge \hat \ell_{k-1}$.
 For the sake of future reference, we note that 
\begin{equation}\label{keynk}
  \phi_*  \ge \hat \ell_k \ge \inf_{u \in \dom h} {\cal A}_k^a (u).
 \end{equation}

Third,  if the test inequality in step 1 holds, then the first component $\hat y_k$ of the pair
 output by Ad-GPB satisfies $\phi(\hat y_k) - \phi_* \le \varepsilon$
 due to the fact that $\hat \ell_k \le \phi_*$.

    Fourth, the cycle-stopping criterion, i.e., the inequality $t_j \le \delta_k$ in step 4, 
is a relaxation of the one used
by the GPB method of \cite{liang2024unified}, in the sense that its right-hand side $\delta_k$ has the extra term $\beta_{k}[\phi(\hat y_{k-1})-\hat \ell_{k-1}]$
involving the relaxation factor $\beta_{k}$. The addition of this term
allows earlier cycles to terminate in fewer inner iterations,
and hence speeds up the overall performance of the method.
Moreover, the definition of $\delta_{k}$ in \eqref{def:deltakj} is similar to that in \eqref{def:deltaj}, except that the constant $\beta$ in \eqref{def:deltaj} is replaced by $\beta_k$ in \eqref{def:deltakj} and $\phi_*$ in \eqref{def:deltaj} is replaced by its lower bound $\hat \ell_{k-1}$ in \eqref{def:deltakj}. The quantities in \eqref{def:gk}  are used to update  $\beta_{k}$ at the end of the $k$-th cycle (see step 5.d of Ad-GPB).

Finally, assume that $\phi_*$ is known and $\hat \ell_0$ is chosen as $\hat \ell_0 = \phi_*$, instead of as in \eqref{def:hatell0}. Then, for every $k \ge 0$, \eqref{def:nk} and the first inequality in \eqref{keynk} imply that  $\hat \ell_k = \phi_*$, and hence that 
 $\hat g_k = 0$, because of \eqref{def:gk}. Consequently, it follows from the test inequality in step~5.d of Ad-GPB that $\beta_{k+1}=\beta_k$ for all $k \ge 1$. Thus,  
 Ad-GPB* with $\beta = 1/4$ and $\lam_{i_k} = \hat \lam_{k-1}$ can be viewed as the variant of Ad-GPB
 in which $\beta_1=1/4$ and $\hat \ell_0=\phi_*$.

We now state the main complexity result for Ad-GPB.
\begin{theorem}\label{main}
 Define 
\begin{align}
\bar K(\varepsilon)&:=
  \left \lceil 8\left(D^2+\hat \lam_0(\phi(\hat x_0)-\hat \ell_0)\right)\frac{Q(\varepsilon)}{\varepsilon^2}+ \log_2^+ \left \{\frac{2(\phi_*-\hat \ell_0)}{\varepsilon}\right \}+1 \right \rceil\label{defbarK}	    
 \end{align}
   where $Q(\varepsilon)$ is as in \eqref{def:bar t},  the tuple $(\hat x_0,\varepsilon,\hat \lam_0,\tau)$ is the input to Ad-GPB, and $\hat \ell_0$ is as in step 0 of Ad-GPB.
 Then, the following statements about 
Ad-GPB hold:
\begin{itemize}
    \item [a)] it terminates   in at most
$4\bar K(\varepsilon)$ cycles and its output $(\hat y,\hat \ell)$ satisfies $\phi(\hat y) - \phi_* \le \phi(\hat y) - \hat \ell \le \varepsilon$;
\item [b)] its total number of iterations is bounded by \begin{equation}\label{eq:totalcomp}
   8\bar K(\varepsilon)
    \left(
    \frac{1}{1-\tau}\ln^+ \left(  \frac{\bar t}{\varepsilon} \right)
    +1
    \right)  +\log_2^+ \left( \frac{\hat \lam_0Q(\varepsilon)}{\varepsilon}\right).
   \end{equation}
\end{itemize}

\end{theorem} 

 We now discuss some implications of the complexity bound derived in Theorem \ref{main}. 
 
  \begin{itemize}
 \item
 For the remarks in this item, assume that $\tau^{-1} = {\cal O}(1)$.
In this case, the iteration complexity of Ad-GPB to find an $\varepsilon$-solution of \eqref{eq:ProbIntro} is 
 \begin{equation}\label{ineq:simpliedbound2}
 \tilde {\cal O}\left(\frac{ \left(D^2+\hat \lam_0  [ \phi(\hat x_0) - \hat \ell_0] \right)}{\varepsilon^2}\left(M_f^2+\varepsilon L_f +\frac{\varepsilon}{(1-\tau)\hat \lam_0 }\right)\right).
\end{equation}
Moreover, if 
$\hat \lam_0$ is chosen to satisfy
\[
\hat \lam_0={\Theta}\left(\frac{D^2}{\phi(\hat x_0)-\hat \ell_0}\right),
\]
then the  complexity of Ad-GPB becomes
\begin{equation}\label{def:comp3}
 \tilde {\cal O}\left(\frac{ D^2}{\varepsilon^2}\left(M_f^2+\varepsilon L_f +\frac{(\phi(\hat x_0)-\hat \ell_0)\varepsilon}{(1-\tau)D^2}\right)\right).
\end{equation}
Finally, in terms of $\varepsilon$ only, if $(1-\tau)^{-1} =   {\cal O}(\varepsilon^{-1})$, then bound
\eqref{def:comp3} is $\tilde {\cal O}(\varepsilon^{-2})$;

\item For the remarks in this item, assume that
$\max\{ \tau^{-1}, (1-\tau)^{-1} \} =  {\cal O}(1)$. Another way of choosing $\hat \lam_0$ is 
a multiple of a Polyak-type stepsize, i.e.,
\[
\hat \lam_0= \frac{\alpha(\phi(\hat x_0)-\hat \ell_0)}{\|\phi'(\hat x_0)\|^2}
\]
where  the multiple factor $\alpha$ is such that $0<\alpha = {\cal O}(1)$ and $\phi'(\hat x_0)$ is a subgradient of $\phi$ at $\hat x_0$. 
Noting that the definition of $\ell_{\phi}$ in \eqref{def:ellphi} and the fact that $\hat \ell_0 = \inf_{u \in \dom h} \ell_{\phi}(u; \hat x_{0})$ imply that 
\begin{align*}
\phi(\hat x_0)-\hat \ell_0 &= \phi(\hat x_0) - \inf_{u \in \dom h} \left(f(\hat x_0)+\inner{f'(\hat x_0)}{u-\hat x_0}+h(u)\right) \\&=  h(\hat x_0) - \inf_{u \in \dom h} \left(\inner{f'(\hat x_0)}{u-\hat x_0}+h(u)\right)\\
&=   \sup_{u \in \dom h} \left(\inner{f'(\hat x_0)}{\hat x_0-u}+h(\hat x_0) - h(u)\right) \\
& \le \sup_{u \in \dom h} \left(\inner{f'(\hat x_0)+h'(\hat x_0)}{\hat x_0-u}\right) \le \|\phi'(\hat x_0)\|D \, 
\end{align*}
where the second last inequality follows from subgradient inequality and the last inequality follows from Cauchy--Schwarz inequality and Assumption A1.
The above inequality and the choice of $\hat \lam_0$ then imply that 
\begin{align*}
\hat \lam_0 [\phi(\hat x_0)-\hat \ell_0] = \frac{\alpha[\phi(\hat x_0)-\hat \ell_0]^2}{\|\phi'(\hat x_0)\|^2}  \le  \alpha  D^2. 
\end{align*}
Thus,  assuming without loss of generality that $\phi(\hat x_0)-\hat \ell_0 > \varepsilon$, the complexity bound \eqref{ineq:simpliedbound2}  becomes
 \begin{equation}\label{def:comp4}
 \tilde {\cal O}\left(\frac{ D^2 }{\varepsilon^2}\left(M_f^2+\varepsilon L_f+\|\phi'(\hat x_0)\|^2\right)\right).
 \end{equation}
  Finally, in terms of $\varepsilon$ only, the bound \eqref{def:comp4} is $\tilde {\cal O}(\varepsilon^{-2})$.
\end{itemize}
The remaining of the section is devoted to proving Theorem \ref{main}. The next result describes properties of Ad-GPB that are similar to the ones derived for Ad-GPB*. Specifically,
its statement a) is analogous to Proposition~\ref{inner}, while its statements b) and c) are analogous to statements a) and b) of Lemma~\ref{lem:t1}.

\begin{lemma}\label{lem:similar}
The following statements hold for every cycle index $k \ge 1$ generated by Ad-GPB:
    \begin{itemize}
     \item [a)] the size of the $k$-th cycle is bounded by
    $|{\cal C}_k| \le s_k + \bar N(\varepsilon)+1 $ where
    $s_k$ denotes the number of bad iterations within it and $\bar N(\cdot)$ is defined as in \eqref{defNs};
\item[b)]  $\hat \lam_k \ge \varepsilon/Q(\varepsilon)$ where $Q(\varepsilon)$ is as in \eqref{def:bar t}; as a consequence, $\sum_{l=1}^k \hat \lam_l \ge\varepsilon k Q(\varepsilon)^{-1}$;
\item[c)] $\sum_{l=1}^k s_l \le \log_2^+(\hat \lam_0Q(\varepsilon)\varepsilon^{-1})$ where
    $s_l$ denotes the number of bad iterations within cycle $l$.
   
\end{itemize}
\end{lemma}

\begin{proof}
a) The proof of a) follows from Proposition~\ref{inner} and the remark following it.

     b) The arguments used for proving this statement are the same as those used in the proof of Lemma \ref{lem:t1}(a).
     
     c)  Recalling that $s_l$ denotes the number of bad iterations in the $l$-th cycle. Observe that for any $j \in \{i_l,\ldots,j_l-1\}$, we have $\lam_{j+1}/\lam_j$ is equal to $1/2$ (resp., equal to $1$) if $j$ is a bad (resp., good) iteration. Using this observation and  the fact that $\lam_{i_l} = \hat \lam_{l-1}$ in view of step 2 of Ad-GPB, we easily see that 
$\hat \lam_l/\hat \lam_{l-1} = (1/2)^{s_l}$, or equivalently, $\log_2 \hat \lam_{l-1} - \log_2 \hat \lam_l = s_l$.
The statement  now follows 
by summing the last equality from $l=1$ to $k$ and using  b).
\end{proof}
It follows from Lemma \ref{lem:similar}(c) and the definition of $Q(\varepsilon)$ in \eqref{def:bar t} that the overall number of bad iterations is 
\[
{\cal O}\left( \log_2\left((1-\tau)\left(\frac{M_f^2}{\varepsilon}+L_f\right)\right)\right).
\]

The next result describes some 
basic facts about the sextuple
$(\hat \lambda_k,\hat x_k, \hat y_k, \hat \Gamma_k,\hat m_k,\hat t_k)$
 obtained in step 5 of Ad-GPB, and hence at the last iteration of its $k$-th cycle. 

\begin{lemma}\label{lem:iterate1}
		For every cycle index $k$ of Ad-GPB, the following statements hold: 
		\begin{itemize}
   \item[a)] we have
      \begin{equation}\label{eq:uppertk}
       \hat t_k \le \beta_{k}[\phi(\hat y_{k-1}) - \hat \ell_{k-1}]+\frac{\varepsilon}{4};
       \end{equation}
   \item[b)] $\hat x_k,\hat y_k \in \dom h$, $\phi_*\le \phi(\hat y_k)\le \hat \Phi_k$, and $\phi(\hat y_k) \le \phi(\hat y_{k-1})$;
   \item [c)] $\hat \ell_k \ge \hat \ell_{k-1}$, $\beta_{k+1} \le \beta_{k} \le 1/4$, and  $ \hat \lam_{k+1} \le  \hat \lam_{k} $;
   \item [d)]
    for  every  $ u\in \R^n$, we have
   \begin{equation}\label{ineq:recu1}
\phi(\hat y_k) -  {\cal A}_k(u) \le
\hat t_k + \frac1{2\hat \lam_k} \left[ \|u-\hat x_{k-1}\|^2 - \|u-\hat x_{k}\|^2 \right].
\end{equation}
 
   \end{itemize}
\end{lemma}
\begin{proof}
		a)  This statement  follows from the fact that the cycle-stopping criterion in the first line of step 4 of Ad-GPB is satisfied at iteration $j_k$, the definition of $\delta_{k}$ in \eqref{def:deltakj}, and the definitions of the quantities  
  $\hat t_k$ and $\hat y_k$ in step 5.b of Ad-GPB.
		
  b) First, $\phi_* \le \phi(\hat y_k)$ obviously holds. The rest of the statement  follows from the fact that $x_j,y_j \in \dom h$ (see first two remarks in the paragraph containing \eqref{eq:minseq}), the definition of $\hat \Phi_k$ in \eqref{def:gk}, and the fact that
  $\hat x_k$ (resp., $\hat y_k$) is the last $x_j$ (resp., $y_j$)
  generated within the $k$-th cycle. 
  
  c) The first inequality in (c) follows from the definition  $\hat \ell_k$ in \eqref{def:nk}. Moreover, the update rule for $\beta_{k+1}$ in step~5.d of Ad-GPB ensures that $\beta_{k+1} \le \beta_k$. Consequently, $\beta_k \le \beta_1 = 1/4$, in view of step 0 of Ad-GPB. Finally, the update rule for $\lambda_j$ in steps~2 and 4.a of Ad-GPB, together with the definition of $\hat \lambda_k$ in step~5.b, implies that $\hat \lambda_{k+1} \le \hat \lambda_k$. 
  
  d) Observe that the definitions of the affine function ${\cal A}_k$ in \eqref{def:hatgamma} and  the quantities  
  $\hat x_k$, $\hat m_k$, $\hat \Gamma_k$, and $\hat \lam_k$, in step 5 of Ad-GPB,  imply that $(\hat x_k,\hat m_k)$ is the pair of optimal solution and optimal value of	\begin{equation}		
			\min \left\lbrace  {\cal A}_k(u) + \frac{1}{2\hat \lam_k}\|u- \hat x_{k-1} \|^2: u\in\R^n \right\rbrace.
			\end{equation}
   The above observation and the fact that $  {\cal A}_k(\cdot) + 1/(2\hat \lam_k)\|\cdot- \hat x_{k-1} \|^2$ is $  (1/\hat \lam_k)$-strongly convex imply that 
		 for any $ u\in \R^n $,
\begin{equation}
		    \hat m_k + \frac{1}{2\hat \lam_k} \|u-\hat x_{k}\|^2
		\le  {\cal A}_k(u) + \frac1{2\hat \lam_k}\|u-\hat x_{k-1}\|^2.
	\end{equation}
	The above relation together with the definition of $\hat t_k$ in step 5.b of Ad-GPB then imply that 
\begin{align*}
         \phi(\hat y_k) - {\cal A}_k(u) +\frac{1}{2\hat\lam_k}\| u -\hat x_{k} \|^2
			&\le \phi(\hat y_k) - \hat m_k + \frac1{2\hat \lam_k}\|u-\hat x_{k-1}\|^2
			=\hat t_k  + \frac{1}{2\hat \lam_k}\|u-\hat x_{k-1}\|^2 
		\end{align*}
   for every $u \in \R^n$, and hence that   statement (d) holds.
  \end{proof}
The following lemma, which is analogous to Lemma \ref{lem:outer:rec}, establishes a preliminary bound on $\hat g_k$  and the $k$-th dual gap $\hat \Phi_k - \hat \ell_k$  for Ad-GPB. 

\begin{lemma}\label{lem:outerru}
For every cycle index $k$ of Ad-GPB, we have:
 \begin{equation}\label{outer:rec2}
\hat \Phi_k - \hat \ell_k  \le   \frac{ \hat \lam_0 [\phi(\hat x_0) - \hat \ell_{0}]}{\hat \Lambda_k}+ \frac{D^2}{\hat \Lambda_k} +\frac{\varepsilon}{2}   + 2\hat g_k ,
 \end{equation}
\begin{equation}\label{ineq:uppergk}
\hat g_k  \le \beta_{\lceil \frac{k}{2} \rceil }\left(\phi_*- \hat \ell_0\right),
\end{equation}
where $D$ is as in \eqref{defD} and  $\hat \Phi_k$ is as in \eqref{def:gk}.
\end{lemma}
\begin{proof}
 We first prove \eqref{outer:rec2}. Multiplying \eqref{ineq:recu1} by $\hat \lam_l$ and using relation \eqref{eq:uppertk}, it follows that for every $u \in \R^n$,
 \begin{align}
  &2\hat \lam_l [ \phi(\hat y_l) - {\cal A}_l( u)] 
\overset{\eqref{ineq:recu1}}\le  2\hat \lam_l \hat t_l  + \|\hat x_{l-1}-u\|^2 -\|\hat x_{l} - u\|^2 \nonumber\\
 &\overset{\eqref{eq:uppertk}}\le 2\hat \lam_l \left[\beta_l\left( \phi(\hat y_{l-1}) - \hat \ell_{l-1} \right)+\frac{\varepsilon}{4}\right]+ \| \hat x_{l-1}- u\|^2 -\|\hat x_{l} - u\|^2 \nonumber \\
 &=   2\hat \lam_{l} \beta_l\left( \phi(\hat y_{l-1}) - \hat \ell_{k} \right)+ 2\hat \lam_{l} \beta_l\left( \hat \ell_k - \hat \ell_{l-1} \right)+ \frac{\varepsilon}{2} \hat \lam_l + \|\hat x_{l-1}-u\|^2 -\|\hat x_{l} - u\|^2 \nonumber\\&\le 
 \hat \lam_{l-1} \left( \phi(\hat y_{l-1}) - \hat \ell_{k} \right)+2 \hat \lam_{l} \beta_l\left( \hat \ell_k - \hat \ell_{l-1} \right)+ \frac{\varepsilon}{2} \hat \lam_l + \|\hat x_{l-1}-u\|^2 -\|\hat x_{l} - u\|^2 \label{iterm:1}
    \end{align}
    where the last inequality uses the fact that \eqref{keynk} and Lemma \ref{lem:iterate1}(b) imply that $\phi(\hat y_{l-1}) - \hat \ell_{k}  \ge 0$ and the inequalities $ \beta_l \le 1/4 \le 1/2$  and $\hat \lam_{l+1} \le  \hat \lam_{l}$ for every $l\ge 1$ (see Lemma \ref{lem:iterate1}(c)).
   Let $u_k^*$ be an optimal solution of $\min \{ {\cal A}_k^a(u) : u \in \cl( {\dom h}) \} $, which exists in view of the assumption that $\dom h$ is bounded and the fact that ${\cal A}_k^a$ is a linear function. Summing 
  the above inequality from $l=\lceil k/2 \rceil$ to $l=k$  with $u = u_k^*$, and using relation  \eqref{keynk} and the definitions of ${\cal A}_k^a$ and $\hat g_k$ in \eqref{def:nk} and \eqref{def:gk}, respectively, we have
  \begin{align*}
     2\sum_{l=\lceil k/2 \rceil}^k  \hat \lam_l [ \phi(\hat y_l) - \hat \ell_k] &\overset{\eqref{keynk}}\le 2 \sum_{l=\lceil k/2 \rceil}^k  \hat \lam_l [ \phi(\hat y_l) - {\cal A}_k^a( u_k^*)]\overset{\eqref{def:nk}} =2  \sum_{l=\lceil k/2 \rceil}^k  \hat \lam_l [ \phi(\hat y_l) - {\cal A}_l( u_k^*)]\\& \overset{\eqref{iterm:1},\eqref{def:hatgamma}}\le
       \sum_{l=\lceil k/2 \rceil}^k\hat \lam_{l-1}  \left [\phi(\hat y_{l-1}) - \hat \ell_k  \right] + 2
 \hat \Lambda_k \hat g_k  + \frac{\hat \Lambda_k \varepsilon}{2}+\|\hat x_{\lceil k/2 \rceil-1}-u_k^*\|^2- \|\hat x_{ k}-u_k^*\|^2 \\
 & \le  \sum_{l=\lceil k/2 \rceil}^k\hat \lam_{l-1}  \left [\phi(\hat y_{l-1}) - \hat \ell_k  \right] + 
 2\hat \Lambda_k \hat g_k  +  \frac{\hat \Lambda_k \varepsilon}{2}+ D^2
  \end{align*}
  where the last inequality follows from the definition of   $D$ in \eqref{defD} and the fact that $u_k^* \in \cl({\dom h})$. 
  After simple algebraic manipulation,  the above inequality implies that  
  \begin{align}
       \sum_{l=\lceil k/2 \rceil}^k  \hat \lam_l [ \phi(\hat y_l) - \hat \ell_k] &\le\hat \lam_{\lceil k/2 \rceil-1} [ \phi(\hat y_{\lceil k/2 \rceil-1}) - \hat \ell_k]+ 
 2\hat \Lambda_k \hat g_k + \frac{\hat \Lambda_k \varepsilon}{2}+D^2 \nonumber \\&\le 
 \hat \lam_0 [ \phi(\hat x_{0}) - \hat \ell_0]+ 
 2\hat \Lambda_k \hat g_k  + \frac{\hat \Lambda_k \varepsilon}{2}+D^2 \label{ineq:lamnon}
  \end{align}
  where the last inequality  uses the facts that the sequences $\{\hat \lam_k\}$ and $\{\phi(\hat y_k)\}$ are non-increasing, and $\{\hat \ell_k\}$ is non-decreasing, 
    in view of  statements (b) and (c) of Lemma \ref{lem:iterate1}.
  Inequality \eqref{outer:rec2} now follows by dividing the above inequality by $\hat \Lambda_k$ and using the definitions of $\hat \Lambda_k$ and $\hat \Phi_k$ in \eqref{def:gk}.

  Now we prove \eqref{ineq:uppergk}. Using the definition of $\hat g_k$ in \eqref{def:gk} and  the fact that $\beta_l \le \beta_{\lceil k/2 \rceil}$ for every $l \ge \lceil k/2 \rceil$  due to Lemma \ref{lem:iterate1}(c), we conclude that 
\begin{align*}
   \hat g_k &\le \frac {\beta_{\lceil k/2 \rceil}}{\hat \Lambda_k} \sum_{l =\lceil k/2 \rceil}^k \hat \lam_l \left [ \hat \ell_k - \hat \ell_{l-1} \right ] \le \beta_{\lceil \frac{k}{2} \rceil}\left[ \phi_* - \hat \ell_0 \right],
\end{align*}
where the last inequality is due to \eqref{keynk} and statement (c) of Lemma \ref{lem:iterate1}, and hence that
\eqref{ineq:uppergk} holds.
\end{proof}

We now provide some remarks about Lemma~\ref{lem:outerru}. 

First, we discuss the role of Step~5.d of Ad-GPB in light of this result. Relation \eqref{outer:rec2} is analogous to \eqref{outer:rec}, except that $\|\hat x_0-x^*\|$ in \eqref{outer:rec} is upper bounded by the diameter $D$ in \eqref{outer:rec2}, and most importantly, the additional term $\hat g_k$  is added to \eqref{outer:rec2}. 
Except this term,
all other terms on the right-hand side of \eqref{outer:rec2} can be bounded in the same way as in \eqref{outer:rec}. Hence, the issue is how to nicely bound $\hat g_k$.
 To this end,  we will use \eqref{ineq:uppergk} to show that the strategy  described in step 5.d of Ad-GPB, i.e.,
 halving 
 $\beta_{k+1}$ whenever $\hat g_k$ is too large compared to the primal-dual gap $ \hat \Phi_k-\hat \ell_k$, nicely bounds  $\hat g_k$ in the long term,  in a delayed manner, and hence
  the primal-dual gap too.

  Finally, in contrast to Ad-GPB* which allows $\lam_{i_k}$ to be chosen in the interval $[\hat \lam_{k-1}, 2 \hat \lam_{k-1}]$ (see  \eqref{obs:lamik}),
  Ad-GPB chooses $\lam_{i_k} = \hat \lam_{k-1} $, which implies that $\{\hat \lam_k\}$ is non-increasing (see Lemma \ref{lem:iterate1}(c)).
  This property, which is used in the second inequality of \eqref{ineq:lamnon},
  makes the above proof simpler. 
However, it is possible to develop a variant of Ad-GPB that chooses $\lam_{i_k}$ in the interval $[\hat \lam_{k-1}, 2 \hat \lam_{k-1}]$ 
subject to the condition that, for some constant $C>0$, the inequality $\hat \lam_k [ \phi(\hat x_{k}) - \hat \ell_k]
\le C \max \{ \hat \lam_0 [ \phi(\hat x_{0}) - \hat \ell_0],D^2 \}$ holds for every $k \ge 1$. It can be shown that such a choice of $\lam_{i_k}$ allows us to extend  Lemma~\ref{lem:outerru} and subsequent results to obtain a complete analysis of the variant outlined above. However, for the sake of shortness and simplicity, we omit the details.


Now we state a result that provides key bounds on the primal-dual gap $\hat \Phi_k-\hat \ell_k$.

\begin{lemma}\label{lem:typeA}
    For every cycle index $k$ of Ad-GPB, the following statements hold:
    \begin{itemize}
        \item[a)]
       if $\beta_{k+1} = \beta_{k} $,
 then we have 
\begin{equation}\label{typeA}
  \hat \Phi_k-\hat \ell_k \le \frac{ 4\hat \lam_0 [\phi(\hat x_0) - \hat \ell_{0}]}{3\sum_{l=\lceil k/2 \rceil}^k \hat \lam_l}+ \frac{4D^2}{3\sum_{l=\lceil k/2 \rceil}^k \hat \lam_l}+\frac{2\varepsilon}{3};
  \end{equation}
    \item [b)]     if $\beta_{k+1} =\beta_{k}/2$, then  we have 
    \begin{equation}\label{typeB}
\hat \Phi_k-\hat \ell_k \le 8\hat g_k , 
\end{equation}
\end{itemize}
where $\hat \Phi_k$ is as in \eqref{def:gk}.
    \end{lemma}
    \begin{proof}
a) The update rule for $\beta_{k+1}$ in step 5.d of Ad-GPB and the assumption that $\beta_{k+1}=\beta_{k}$ imply that $\hat g_k \le (\hat \Phi_k-\hat \ell_k)/8$. This observation and
inequality \eqref{outer:rec2} then
immediately imply \eqref{typeA}.

b) Statement b) follows from the assumption that $\beta_{k
+1}=\beta_{k}/2$ 
and the update rule for $\beta_{k+1}$ in step 5.d of Ad-GPB.
    \end{proof}

We now provide some remarks on Lemma~\ref{lem:typeA}. First, inequality \eqref{typeA} is analogous to \eqref{outer:rec}. Second, if \eqref{typeA} holds at every iteration, then it can be shown, using similar arguments to those in \S\ref{outer:Ad-GPB}, that  the number of cycles generated by Ad-GPB can be nicely bounded. However, since \eqref{typeB} may also occur, the analysis in this section (especially, the proof of Theorem \ref{main} below) is more involved than that   for Ad-GPB* in \S\ref{outer:Ad-GPB}.

\vgap
We are now ready to prove Theorem \ref{main}.

\noindent{\bf Proof of Theorem \ref{main}}
a) To simplify notation, let
$\bar K = \bar K(\varepsilon)$.
It is easy to see that
\beq \label{coroN}
\frac{D^2+\hat \lam_0(\phi(\hat x_0)-\hat \ell_0)}{\bar K} \frac{Q(\varepsilon)}{\varepsilon}\le \frac{\varepsilon}{8}, \qquad \frac{1}{2^{\bar K-3}}\left(\phi_*-\hat \ell_0\right)<\frac{\varepsilon}{\beta_1 }.
\eeq 

We first prove that Ad-GPB finds an iterate $\hat y_k$ satisfying $\phi(\hat y_k) - \hat \ell_k \le \varepsilon$ in at most
$4\bar K$ cycles. Suppose for contradiction that Ad-GPB generates a cycle $K \ge 4\bar K+1$. Since the Ad-GPB did not stop at cycles
from $1$ to $K-1$, we have that
   \begin{equation}\label{notfinished}\phi(\hat y_k)-\hat \ell_k > \varepsilon \quad \forall k\in \{1,\ldots,4\bar K\}.\end{equation} 
   We then have that 
   \begin{equation}\label{halfbeta}
       \beta_{k+1}=\frac{\beta_{k}}2 \quad \forall k \in \{\bar K,\ldots,4\bar K\}
       \end{equation} 
since otherwise we would have some
$\beta_{k+1}=\beta_{k}$ for some
$k \in \{\bar K,\ldots,4\bar K\}$,
and this together with \eqref{notfinished} and \eqref{typeA}, Lemma \ref{lem:similar}(b), and
 Lemma \ref{lem:iterate1}(b),
would yield the contradiction that
\begin{align*}
\varepsilon< \phi(\hat y_k) - \hat \ell_k &\overset{\textbf{L.} \ref{lem:iterate1}(b)}\le  \hat \Phi_k-\hat \ell_k \overset{\eqref{typeA}}\le   \frac{4\left(D^2+\hat \lam_0(\phi(\hat x_0)-\hat \ell_0)\right)}{3\sum_{l=\lceil k/2 \rceil}^k \hat \lam_l}
 +\frac{2\varepsilon}{3} \overset{\textbf{L.} \ref{lem:similar}(b)} \le
 \frac{8\left(D^2+\hat \lam_0(\phi(\hat x_0)-\hat \ell_0)\right)Q(\varepsilon)}{3 k  \varepsilon}+\frac{2\varepsilon}{3} \\
 & \ \ \ \ \le
 \frac{8\left(D^2+\hat \lam_0(\phi(\hat x_0)-\hat \ell_0)\right)Q(\varepsilon)}{3\bar K \varepsilon} +\frac{2\varepsilon}{3} \overset{\eqref{coroN}} 
 \le \frac{\varepsilon}3 +\frac{2\varepsilon}{3} = \varepsilon,
\end{align*}
where the last inequality is due to the first inequality in \eqref{coroN}.
Relations \eqref{ineq:uppergk}, \eqref{typeB},
\eqref{notfinished} and \eqref{halfbeta},
and Lemma \ref{lem:iterate1}(b), all with
$k=4\bar K$, then yield 
\[ 
\varepsilon \overset{\eqref{notfinished}} < \phi(\hat y_{4\bar K})-\hat \ell_{4\bar K} \overset{\textbf{L.} \ref{lem:iterate1}(b)}\le \hat \Phi_{4 \bar K}-\hat \ell_{4\bar K} \overset{\eqref{typeB}}\le 8 \hat g_{4\bar K}\overset{\eqref{ineq:uppergk}} \le 8 \beta_{ 2\bar K
 } \left(\phi_*-\hat \ell_0\right)
\overset{\eqref{halfbeta}}\le\frac{1}{2^{\bar K-3}} \beta_{\bar K}\left(\phi_*-\hat \ell_0\right) \overset{\eqref{coroN}}<  \frac{\beta_{\bar K}}{\beta_1} \varepsilon 
\]
where 
the last inequality is due to the second inequality in \eqref{coroN}.
Since $\beta_{\bar K} \le \beta_1$, the above inequality gives the desired contradiction, and hence a) holds. The fact that it outputs an iterate $\hat y$ satisfying $\phi(\hat y) - \hat \ell \le \varepsilon$ follows from the termination criterion in step 1 of Ad-GPB.

b) Let $\bar k  \le 4\bar K$ denote the numbers of cycles generated by Ad-GPB.
Statements (a) and (c)  of Lemma \ref{lem:similar}  then imply that the total number of iterations that Ad-GPB terminates is bounded by
    \begin{align*}
    \sum_{k=1}^{\bar k} &|{\cal C}_k| \le\sum_{k=1}^{\bar k} \left(
    \frac{2}{1-\tau}\ln^+ \left(  \frac{\bar t}{\varepsilon} \right) 
   +2+s_k\right)\\
 &\le \log_2^+  \frac{\hat \lam_0 Q(\varepsilon)}{\varepsilon} + 4\bar K\left(
    \frac{2}{1-\tau}\ln^+ \left(  \frac{\bar t}{\varepsilon} \right) 
    +2
    \right)
 \end{align*}
and hence bounded by  \eqref{eq:totalcomp}, due to  the fact that $1 \le 1/(1-\tau)$.
\endproof
\section{A Restarted Ad-GPB* Variant for Strongly Convex Problems}\label{sec:strongconvex}
Theorem~\ref{main1} shows that the iteration-complexity bound for
Ad-GPB* is $\widetilde{\cal O}(\varepsilon^{-2})$. This section considers
the more general setting where $\phi$ is $\mu$-convex for some $\mu \ge 0$, and presents a restarted variant of Ad-GPB* which, similar to 
Ad-GPB*, is parameter-free and line-search-free, and
has an
$\widetilde{\cal O}(1/(\mu\varepsilon))$
iteration-
 complexity bound when $\mu>0$.

This section  assumes the validity of assumptions (A1)-(A3) in
\S \ref{Sec:Ass}, assumption \eqref{keyass}, and 
the following new assumption:
\begin{align*}
    \boxed{ \mbox{function $\phi$ is $\mu$-convex for some $\mu \ge 0$.}}
\end{align*}
Thus, strong convexity is imposed on the whole objective function of \eqref{eq:ProbIntro}, rather than on either component function $f$ or $h$.

We start by describing 
Rs-Ad-GPB*, a restart version of Ad-GPB*. 

\noindent\rule[0.5ex]
{1\columnwidth}{1pt}
	
	 Rs-Ad-GPB*
	
	\noindent\rule[0.5ex]{1\columnwidth}{1pt}
 \textbf{Input:} $(\hat z_0,\varepsilon) \in \dom h \times  \R_{++}$, $(\hat \lam_0,\tau,\beta) \in \R_{++} \times (0,1) \times [0,1/4];$

     \begin{itemize}
     \item [0.] set $s =1$;
     \item [1.] if $\phi(\hat z_{s-1})-\phi_* \le \varepsilon$,  then output $\hat z = \hat z_{s-1}$ and \textbf{stop},  else set \begin{equation}\label{def:vars}
         \varepsilon_{s-1}  = \frac{\varepsilon}2 + \frac{\phi(\hat z_{s-1})-\phi_*}4;
         \end{equation}
         \item [2.] call Ad-GPB*($ \hat z_{s-1}, \varepsilon_{s-1}, \hat \lam_0,\tau,\beta$) to obtain output $\hat z_{s}$; 
         \item [3.]   set $s \rightarrow s+1$ and go to step 1.
     \end{itemize}
	
	\noindent\rule[0.5ex]{1\columnwidth}{1pt}
 We now introduce some terminology related to Rs-Ad-GPB*.  
Each call to Ad-GPB\({}^*\) in Step~2 is called a stage. The index $s$ counts the stages.

The next lemma presents some basic results about Rs-Ad-GPB*.
\begin{lemma}\label{lem:basicRsAdGPB}
    For every stage index $s\ge1$ of Rs-Ad-GPB*, the following statements hold:
    \begin{itemize}
        \item [a)]$\phi(\hat z_{s}) \le \phi(\hat z_{s-1}) \le \phi(\hat z_0)$;
        \item [b)] $\phi(\hat z_{s} ) - \phi_* \le \varepsilon/2 + (\phi(\hat z_{s-1})-\phi_*)/4 $.
    \end{itemize}
\end{lemma}

\begin{proof}
   a)  The first inequality follows from step 2 of Rs-Ad-GPB* and the first inequality in \eqref{Outputy} with
\((\hat x_0,\hat x)=(\hat z_{s-1},\hat z_s)\). The second inequality follows immediately from the first one.
   
   b) This statement immediately follows from step 2 of Rs-Ad-GPB*, the second inequality in \eqref{Outputy} with
\(\hat x=\hat z_s\) and $\varepsilon = \varepsilon_{s-1}$, and the definition of $\varepsilon_{s-1}$ in step 1 of Rs-Ad-GPB*.
\end{proof}

We are now ready to present the main result of this section.

\begin{theorem}\label{thm:strongconvex}
Define 
\begin{equation}\label{def:finalKmu}
\hat K_\mu: = 
\left \lceil \frac{32Q(\varepsilon/2)}{\varepsilon}
\min\left\{
 \frac{2(1+\mu\beta\hat\lambda_0)}{\mu},
    \;
    \frac{D^2+\beta\hat\lambda_0
    \bigl[\phi(\hat z_0)-\phi_*\bigr]}{\varepsilon}
\right\}\right \rceil,
\end{equation}
 where  $Q(\cdot)$ is as in \eqref{def:bar t}. Then, the following statements about Rs-Ad-GPB* hold: 
 \begin{itemize}
    \item [a)] the total number of iterations performed by
      any of its stages is bounded by
     \[
     \left\{     \frac{3 \hat K_{\mu}}{1-\tau} \left[
   \ln^+ \left(  \frac{2\bar t}{\varepsilon} \right)
    +1
    \right]  +\log_2^+ \left( \frac{2\hat \lam_0 Q(\varepsilon/2)}{\varepsilon}\right)\right\};
     \]
     \item [b)] it outputs a point $\hat  z$ satisfying $\phi(\hat z) - \phi_* \le \varepsilon$ in a number of stages bounded by \begin{equation}\label{def:S}
      S:=\lceil \log_2^+ \left(3(\phi(\hat z_0)-\phi_*)/\varepsilon \right)\rceil+1   . 
     \end{equation}
      \end{itemize} 
     As a consequence, the total number of iterations performed by Rs-Ad-GPB* is bounded by 
      \begin{equation}\label{fianlcomplexitymu}
  S \left\{     \frac{3 \hat K_{\mu}}{1-\tau} \left[
   \ln^+ \left(  \frac{2\bar t}{\varepsilon} \right)
    +1
    \right]  +\log_2^+ \left( \frac{2\hat \lam_0 Q(\varepsilon/2)}{\varepsilon}\right)\right\}.
    \end{equation}
\end{theorem}
\begin{proof} 
a) Let $s \ge 1$ be a stage index generated by Rs-Ad-GPB*. Using the fact that the definition of $\varepsilon_{s-1}$ in \eqref{def:vars} implies that $\varepsilon_{s-1} = \delta(\varepsilon/2,1/4,\hat z_{s-1})$ where $\delta(\cdot,\cdot,\cdot)$ is as in \eqref{def:delta}, we conclude from Proposition \ref{keyleema:strongly} with $(\varepsilon,\gamma,\hat x_0) = (\varepsilon/2, 1/4,\hat z_{s-1})$ that this stage  terminates in at most 
\[
\frac{3\hat J_{\mu}(\varepsilon/2,1/4,\hat z_{s-1})}{1-\tau} \left[
   \ln^+ \left(  \frac{2\bar t}{\varepsilon} \right)
    +1
    \right]  +\log_2^+ \left( \frac{2\hat \lam_0 Q(\varepsilon/2)}{\varepsilon}\right)
\]
iterations where $\hat J_\mu(\cdot,\cdot,\cdot)$ is as in \eqref{def:Kmu}. Moreover, Assumption~(A1) and
Lemma~\ref{lem:basicRsAdGPB}(a) imply that
\[
d(\hat z_{s-1})\le D,
\qquad
\phi(\hat z_{s-1})\le \phi(\hat z_0),
\]
and hence that $
\hat K_\mu
\ge
\hat J_\mu(\varepsilon/2,1/4,\hat z_{s-1})$
due to the definitions of \(\hat K_\mu\) and
\(\hat J_\mu(\cdot,\cdot,\cdot)\) in \eqref{def:finalKmu} and
\eqref{def:Kmu}, respectively. 
The statement now immediately follows from the above two conclusions.

b)  First note that the definition of $S$ in \eqref{def:S} implies that
\begin{equation}\label{ineq:directS}
\frac{\phi(\hat z_{0})-\phi_*}{4^S} \le \frac{\varepsilon}{3}.
\end{equation}
Suppose for contradiction that Rs-Ad-GPB* does not terminate at
   the end of the $S$-th stage. This implies that 
   $\phi(\hat z_S)-\phi_* > \varepsilon$.  
  On the other hand, Lemma \ref{lem:basicRsAdGPB}(b) for $s=1,\ldots,S$ implies that
\[
  \phi(\hat z_{S} ) - \phi_* \le \frac{\varepsilon}{2} + \frac{\phi(\hat z_{S-1})-\phi_*}{4} \le \frac{\varepsilon}{2}+ \frac{1}{4}\left(\frac{\varepsilon}{2} + \frac{\phi(\hat z_{S-2})-\phi_*}{4}\right) \le \ldots \le \frac{2\varepsilon}{3}+\frac{\phi(\hat z_{0})-\phi_*}{4^S} \overset{\eqref{ineq:directS}}\le \varepsilon,
\]
where the last inequality is due to \eqref{ineq:directS}.
 Since the above inequality contradicts the assumption that $\phi(\hat z_S)-\phi_* >\varepsilon$, statement b) follows.

 The last conclusion of the theorem follows immediately from  a) and b).
\end{proof}

 We now make some remarks about the complexity bound derived in Theorem \ref{thm:strongconvex}. 
 For all the remarks below, it is assumed that $\tau^{-1} = {\cal O}(1)$ and $\mu >0$.
 First,  the definitions of $Q(\cdot)$ and $\hat K_\mu$ in \eqref{def:bar t} and \eqref{def:finalKmu}, respectively, imply that
 the iteration complexity  bound \eqref{fianlcomplexitymu} is 
\begin{equation}\label{simplifiedcombound3}
 \tilde {\cal O}\left(\frac{ \beta \mu  \hat \lam_0+1}{\mu \varepsilon}\left(M_f^2+\varepsilon L_f+\frac{\varepsilon}{(1-\tau)\hat \lam_0 }\right)\right).
\end{equation}
Thus, the condition
\[
\beta\mu\hat\lambda_0={\cal O}(1)
\qquad\text{and}\qquad
\frac{\varepsilon}{(1-\tau)\hat\lambda_0}
={\cal O}(M_f^2+\varepsilon L_f),
\]
implies that \eqref{simplifiedcombound3} reduces to $
 \tilde {\cal O}((M_f^2+\varepsilon L_f) /( \mu \varepsilon))$. 
 
 Second, for given positive scalars $\bar M$ and $\bar \mu$, it is argued in \cite{guigues2026universal} that  $  {\Theta}\left(\bar M^2 /( \bar \mu \varepsilon)\right)$ is the optimal complexity for the class of instances  $(f,h)$ of problem \eqref{eq:ProbIntro} where $L_f=0$, $M_f \le \bar M$, and $\mu \ge \bar \mu$. Moreover, \cite{guigues2026universal} presents two parameter-free algorithms (i.e., which do not require $(M_f,L_f,\mu)$), namely a composite subgradient algorithm and a proximal bundle method, that achieve near-optimal complexity for the above class of instances.
 However, 
  both of these methods may backtrack. In particular, in contrast to
 Rs-Ad-GPB*,  the proximal bundle method of  \cite{guigues2026universal}
  may restart an entire cycle, and hence throw several iterations away.

\section{Computational experiments
}\label{numer}

This section reports the computational results of Ad-GPB* and Ad-GPB against other  bundle methods. 
 It contains three subsections. The first one presents computational results of these methods for solving $l_1$ feasibility problems.  The second one showcases their computational results for solving Lagrangian cut subproblems that are used to generate cuts in the context of integer programming. Finally, the third one presents their computational results for solving constrained $l_1$ feasibility problems.



We now describe all the methods that are used in our benchmark. We start with the ones that require the knowledge of the optimal value.

\begin{itemize}
    \item [1.]
{\em Polyak subgradient (P-Sub*):}
given $x_k$, this method computes
 \[
     x_{k+1} = \argmin_x \left\{\ell_{\phi} (x;x_k)+\frac{1}{2\lam_{\rm pol}(x_k)}\|x-x_k\|^2\right\}
     \]
     where
\begin{equation}\label{lampol}
    \lam_{\rm pol}(x) := \frac{\phi(x)-\phi_*}{\|g(x)\|^2}
\end{equation}
and $g(x) \in \partial f(x)$.

\item [2.] \textit{GPB*:} this method is a special case of the one outlined in \S \ref{GPB}, and formally described in \cite{liang2021proximal, liang2024unified}, where the prox stepsize $\lam$ is set to $\lambda_{\rm pol}(x_0)$. This method requires a parameter $\tau$.


\item [3.] \textit{Ad-GPB*:} a cycle of Ad-GPB* (or Ad-GPB) is called good if the prox stepsize
does not change within it. 
 For any $k\ge 2$, Ad-GPB*
 either sets 
 $\lam_{i_k}=2\hat \lam_{k-1}$
 if all the cycles before the $k$-th one are  good, or else sets $\lam_{i_k}=\hat \lam_{k-1}$. This method sets $\beta = 1/4$ and requires  parameters $\tau$ and $\hat \lam_0$.
 \item [4.]\textit{P-Ad-GPB*:}
 this is a Polyak-type variant of Ad-GPB* where
the initial prox stepsize for the $k$th-cycle
is set to
$\lam_{i_k}=\alpha \lam_{\rm pol}(\hat x_{k-1})$ for $k \ge1 $ and some $\alpha >0$.  This method sets $\beta = 1/4$ and requires a parameter $\tau$. 

\item [5.] \textit{PB*:} this is a classical proximal bundle method with variable stepsize which is described in \cite{kiwiel1995proximal}.  This method requires four parameters, namely, $\beta$, $a$, and $\lam_{min}$, as in \cite{kiwiel1995proximal} and initial stepsize $\hat \lam_0$.

\item [6.] \textit{BG*:} This method is a variant of the classical proximal bundle method of \cite{diaz2023optimal} that sets the stepsize 
$\rho_k$ for the $k$-th cycle as
\begin{equation}\label{stepsize：BG}
\rho_k = \frac{f(\hat x_{k-1})-f_*}{\|\hat x_{k-1}-x^*\|^2}.
\end{equation}
 This method requires parameter $\beta$ as in \cite{diaz2023optimal}.
\item [7.] \textit{APL*:} this is a variant of the Bundle Level (BL) method presented in \cite{lan2015bundle} with $\alpha_k$ chosen as $\alpha_k = 2/(k+1)$. Moreover, it sets the lower bound to the known optimal value rather than updating it by solving a linear program at each iteration. This method  requires two parameters, namely, $\beta$ and $\theta$ as in \cite{lan2015bundle}. 
\end{itemize}

We now describe the initialization and bundle update rules used by the methods. First, methods   3 and 5 choose the initial stepsize  the same as in method 2. Second, methods 2, 3, 4, 5, and 6 update the bundle $\Gamma_k$ using the two-cut scheme described in \S\ref{sec:shadow}. The APL* method also employs a two-cut update, but in a slightly different manner from that described in \S\ref{sec:shadow}. Third, methods  2, 3, and 4 set $\tau = 0.95$. Finally, all the aforementioned algorithms terminate when the primal gap becomes sufficiently small. 
   
We now make several remarks on Algorithms~5--7. 
First, PB* relies on $\phi_*$ in two aspects: (i) its initial stepsize requires knowledge of $\phi_*$, and (ii) its termination criterion depends on $\phi_*$. Second, the stepsize choice \eqref{stepsize：BG} used for BG*  requires knowledge of $x^*$. Even though it is completely impractical, we still use it in our benchmark since the optimal solutions for our randomly generated instances were all known. A more realistic approach would be to replace
$\|\hat x_{k-1}-x^*\|$ by a uniform bound (e.g., a bound on the diameter of $\dom h$), but it was not clear how to obtain such a bound in the setting of the unconstrained problems used in our benchmark.  
Third, each iteration of APL in \cite{lan2015bundle} generates a  lower bound by solving a linear program, which is used  to define a certain level set,
and subsequently the next iterate. In contrast, APL* uses $\phi_*$ as the lower bound, which usually yields a different level set and iterate than those generated by APL.


We now describe the algorithms that do not require knowledge of the optimal value:
\begin{itemize}
    \item [8.] \textit{Ad-GPB:}   
this method is described in \S \ref{Sec:General}.
 It chooses $\tau = 0.95$ and updates the bundle $\Gamma_k$ according to the two-cut scheme described in \S \ref{sec:shadow}.  This method sets $\hat \lam_0 = (\phi(x_0)-\hat \ell_0)/\|g(x_0)\|^2$ where $g(x_0) \in \partial f(x_0)$ and $\hat \ell_0:= \min \{ \tilde \ell_f(x;x_0) : e^Tx \le D, \, x\ge 0\}$. 
 \item [9.] {APL:} this method is described in \cite{lan2015bundle}, which updates the bundle $\Gamma_k$ using two cuts and chooses $\alpha_k$  as $\alpha_k = 2/(k+1)$. This method  requires two parameters, namely, $\beta$ and $\theta$ as in \cite{lan2015bundle}. 
\end{itemize} 

 Both   algorithms 8 and 9 terminate when the primal-dual gap becomes  small.

   All experiments were written in MATLAB 2023a and were performed on PACE \footnote{https://pace.gatech.edu/} with Dual Intel Xeon Gold 6226 CPUs @ 2.7 GHz (24 cores/node). The implementation of our algorithms is available at \url{https://github.com/Honghao-Zhang1/Parameter-free-proximal-bundle-methods-with-adaptive-stepsizes-.git}.
   
We now discuss how the required parameters (if any) for the above algorithms were tuned in
our benchmark.
First, the tuning of the parameter $\alpha$ in P-Ad-GPB* is performed only on small instances of  the \(l_1\) feasibility problem. The resulting choice is then fixed and used uniformly across all other problem instances without further adjustment. Specifically, $\alpha$ is selected from the set $\{1, 20, 40, 60\}$. Among these candidates, $\alpha = 40$ consistently provides the best performance and is therefore used in all reported experiments.

Second, for the parameter tuning of APL and APL*, pairs  $(\beta,\theta)$  such that $\beta \in \{0.3,0.5,0.7,0.9\}$ and $\theta \in \{0.2,0.35,0.5,0.65\}$ were used. For PB*, we used  triples $(\beta,a,\lam_{min})$ such that $\beta \in \{0.1,0.5,0.7\}$, $\lambda_{\min} \in \{10^{-6},10^{-5},10^{-3}\}$, and $a \in \{2,4,5\}$. For  BG*, we used scalars $\beta \in \{0.1,0.5,0.7\}$. For all these methods, the best combination of parameters was selected using small representative instances from each problem class, and then used to solve all instances of the same problem class in our benchmark.

\subsection{$l_1$ feasibility problem}\label{l_1}
The main goal of this subsection is to benchmark the algorithms that require knowledge of $\phi_*$ on the $l_1$ feasibility problem. 


We are now  ready to describe the $l_1$ feasibility problem. The problem can be formulated as: 
\begin{equation}\label{toy}
    f_*:=\min_{x \ge 0} f(x):=\|Ax-b\|_{1}
\end{equation}
where $A \in \R^{m\times n}$ and $b\in \R^{ m}$ are known.  We now describe the way of generating the data. Matrix $A$ is randomly generated in the form  $A=DN$ where the nonzero entries of the sparse matrix $N \in \R^{m\times n} $ are i.i.d.\ sampled from the standard normal $\mathcal{N}(0,1)$ 
distribution and D is a diagonal matrix where the diagonal of $D$ are i.i.d.\ sampled from $\mathcal{U}[0,1000]$. Vector $b$ is then determined as
 $b=Ax^*$ where $x^*=(v_*)^2$ for some
 vector $v_* \in \R^n$ whose entries are i.i.d.\ sampled from the standard Normal distribution $\mathcal{N}(0,1)$. Finally, we generated  $x_0 = (v_0)^2$ for some vector $v_0 \in \R^n$ whose entries are i.i.d.\ sampled from the uniform distribution over $(0,1)$. Clearly, $x^*$ is a global minimizer of \eqref{toy}, whose
 optimal value $f_*$ equals 
 zero.

To check how sensitive GPB* and Ad-GPB*  are relative to the initial choice of prox stepsize $\hat \lam_0$, we test them for  $\hat \lam_0=\alpha \lam_{\rm pol}(x_0)$ where $\alpha \in \{0.01,1,100\}$. 
   The results are presented in Table \ref{tabfirstpb:sparse1}.  
The quantities $\theta_m$, $\theta_n$ and $\theta_s$  are defined as $\theta_{m}  = m/10^3$, $\theta_{n} =  n/10^3$, and $\theta_s := \text{nnz}(A)/mn$, where $\text{nnz}(A)$ is the number of non-zero entries of $A$.  
For each instance, the row labeled ``CPU/it'' reports the total CPU time and the average time per iteration, 
where the latter is scaled by $10^{-3}$. The row labeled ``O/AI'' reports the numbers of outer and average inner iterations, respectively. 
An entry marked $*$ indicates that the method fails to achieve the desired relative accuracy within the time limit; 
in this case, the corresponding ``O/AI'' entry reports the achieved relative accuracy as defined in \eqref{rela-var} below. 
Bold numbers highlight the best-performing method for each instance.

All the methods tested in this subsection are terminated based on the following criterion:
\begin{equation}\label{rela-var}
f(x_k) - f_* \le \bar \varepsilon [ f(x_0)-f_* ]
\end{equation}
where $\bar \varepsilon = 10^{-4}$.

\begin{table}[H]
\centering
\renewcommand{\arraystretch}{1.18}
\setlength{\tabcolsep}{2.6pt}

\begin{tabular}{|c|c|ccc|ccc|}
\hline
\multicolumn{2}{|c|}{ALG.} &
\multicolumn{3}{c|}{GPB*} &
\multicolumn{3}{c|}{Ad-GPB*} \\
\hline
\multicolumn{2}{|c|}{$(\theta_m,\theta_n,\theta_s)\backslash \alpha$} &
$10^{-2}$ & $1$ & $10^2$ &
$10^{-2}$ & $1$ & $10^2$ \\
\hline\hline

\multirow{2}{*}{$(1,20,10^{-2})$}
& CPU/it
& $194/2.84$ & $464/3.02$ & $*$
& $\mathbf{49}/2.28$ & $63/2.10$ & $51/2.25$ \\
\cline{2-8}
& O/AI
& $9924/6.88$ & $102/1507.4$ & $.13e{-03}$
& $10815/2.01$ & $14878/2.01$ & $11229/2.01$ \\
\hline\hline

\multirow{2}{*}{$(3,30,10^{-2})$}
& CPU/it
& $809/5.53$ & $691/5.72$ & $*$
& $347/4.50$ & $317/4.48$ & $\mathbf{304}/4.45$ \\
\cline{2-8}
& O/AI
& $12242/11.95$ & $123/982.65$ & $.13e{-03}$
& $38559/2.00$ & $35342/2.00$ & $34056/2.00$ \\
\hline\hline

\multirow{2}{*}{$(5,50,10^{-2})$}
& CPU/it
& $1561/12.65$ & $1273/12.80$ & $*$
& $671/11.16$ & $543/10.92$ & $\mathbf{496}/10.98$ \\
\cline{2-8}
& O/AI
& $10171/12.13$ & $106/937.76$ & $.13e{-03}$
& $30040/2.00$ & $24795/2.00$ & $22499/2.01$ \\
\hline\hline

\multirow{2}{*}{$(10,100,10^{-3})$}
& CPU/it
& $1558/10.24$ & $1373/10.39$ & $*$
& $663/8.21$ & $\mathbf{629}/8.13$ & $642/8.25$ \\
\cline{2-8}
& O/AI
& $13260/11.48$ & $134/985.78$ & $.13e{-03}$
& $40331/2.00$ & $38684/2.00$ & $38867/2.00$ \\
\hline\hline

\multirow{2}{*}{$(20,200,10^{-3})$}
& CPU/it
& $3472/25.49$ & $2902/26.07$ & $*$
& $1539/22.10$ & $\mathbf{1292}/21.51$ & $1357/21.79$ \\
\cline{2-8}
& O/AI
& $11733/11.61$ & $120/927.34$ & $.13e{-03}$
& $34757/2.00$ & $30032/2.00$ & $31039/2.01$ \\
\hline\hline

\multirow{2}{*}{$(30,300,10^{-3})$}
& CPU/it
& $*$ & $*$ & $*$
& $3210/47.97$ & $\mathbf{2529}/48.38$ & $2917/48.40$ \\
\cline{2-8}
& O/AI
& $.28e{-03}$ & $.16e{-03}$ & $.13e{-03}$
& $33395/2.00$ & $26137/2.00$ & $30040/2.01$ \\
\hline\hline

\end{tabular}
\caption{Numerical results for sparse instances. The stopping criterion
\eqref{rela-var} with $\bar \varepsilon = 10^{-4}$ is used, and a time
limit of 3600 seconds (1 hour) is imposed.}
\label{tabfirstpb:sparse1}
\end{table}
Table \ref{tabfirstpb:sparse1} demonstrates that Ad-GPB* outperforms GPB* and exhibits greater robustness with respect to the initial stepsize. Moreover, Ad-GPB* requires substantially fewer inner iterations than GPB, confirming the theoretical insight that these variants adaptively adjust stepsizes within each cycle in an efficient manner. 

Next, we compare methods that require the optimal value on a larger set of instances. The computational results for ten sparse instances are reported in Table \ref{tabfirstpb:sparse2}, where the target tolerance $\bar{\varepsilon}$ in \eqref{rela-var} is fixed at $10^{-4}$. The quantities $\theta_m$, $\theta_n$, and $\theta_s$ are defined as in Table \ref{tabfirstpb:sparse1}. Each table entry corresponds to the CPU running time in seconds. An entry marked $*/N$ indicates that the computed solution failed to achieve the desired relative accuracy within the time limit, in which case $N$ denotes the relative accuracy of the computed solution as defined in \eqref{rela-var}. Bold numbers highlight the method achieving the best performance for each instance. The time limit for all runs reported in Table \ref{tabfirstpb:sparse2} is one hour.

\renewcommand{\arraystretch}{1}
\begin{table}[H] 
\centering
\footnotesize
\begin{tabular}{|c|c|c|c|c|c|c|c|}
\hline
{$(\theta_m,\theta_n ,\theta_s)$} 
&
{Ad-GPB*}&
{P-Ad-GPB*}&
{P-Sub*}&
{BG*}&
{PB*}&
{GPB*}&
{APL*}\\
\hline

(1,20,$10^{-2}$) 
&63&33&482&*/.45e-02&\textbf{3}&464&*/.30e-02\\
\hline

(2,20,$10^{-2}$) 
&140&\textbf{62}&444&*/.30e-02&835&357&*/.23e-02\\
\hline

(3,30,$10^{-2}$) 
&317&\textbf{128}&1109&*/.31e-02&*/.14e-03&690&*/.19e-02\\
\hline

(4,40,$10^{-2}$) 
&535&\textbf{293}&2744&*/.33e-02&*/.14e-03&1213&*/.17e-02\\
\hline

(5,50,$10^{-2}$) 
&543&\textbf{261}&3276&*/.34e-02&*/.10e-03&1271&*/.15e-02\\
\hline
\hline

(10,100,$10^{-3}$) 
&629&\textbf{335}&2521&*/.35e-02&*/.16e-03&1337&*/.11e-02\\
\hline

(15,150,$10^{-3}$) 
&1024&\textbf{551}&*/.14e-03&*/.38e-02&*/.15e-03&2206&*/.87e-03\\
\hline

(20,200,$10^{-3}$) 
&1292&\textbf{695}&*/.16e-03&*/.39e-02&*/.13e-03&2885&*/.78e-03 \\
\hline

(25,250,$10^{-3}$) 
&1714&\textbf{1018}&*/.23e-03&*/.43e-02&*/.14e-03&*/.12e-03&*/.70e-03\\
\hline

(30,300,$10^{-3}$) 
&2529&\textbf{1493}&*/.27e-03&*/.38e-02&*/.13e-03&*/.15e-03&*/.64e-03\\
\hline

\end{tabular}
\caption{Numerical results for sparse instances. Stopping criterion \eqref{rela-var} with $\bar \varepsilon  = 10^{-4}$ is used and a time limit of
3600 seconds (1 hour) is given. Each entry reports CPU running time in seconds.
An entry marked $*/N$ indicates that the computed solution failed to achieve the desired relative accuracy within the time limit, where $N$ denotes the attained relative accuracy.}
\label{tabfirstpb:sparse2}
\end{table}

    Table \ref{tabfirstpb:sparse2}
 demonstrates that Ad-GPB* and P-Ad-GPB* generally outperform other methods in terms of CPU running time. Additionally, P-Ad-GPB* stands out as a particularly effective variant, outperforming other methods in nine out of ten instances.  

\subsection{Lagrangian cut problem}\label{Lar}

This subsection presents the numerical results comparing Ad-GPB* and P-Ad-GPB* against P-Sub* and APL* on a convex nonsmooth optimization problem that has broad applications in the field of integer programming (see e.g. \cite{zou2019stochastic}). For completeness, the results of  PB* are also included.

The problem considered in this subsection arises in the context of solving 
the stochastic  binary multi-knapsack problem 
described next.
Let 
 $\Xi$ be a set  with a finite number of scenarios, and let 
 $q: \Xi \to \R^n$,  $(d,T,W) \in \R^m \times \R^{m  \times n}  \times \R^{m  \times n}$ and $(b,A) \in \R^l \times \R^{l  \times n}$ be given. The problem is
\begin{align}\label{BKP}
\min_x & \ c^T x+P(x) \\
\text { s.t. } & A x \geq b \nonumber\\
& x \in\{0,1\}^n \nonumber
\end{align}
where $P(x):=\mathbb{E}_{\xi}\left[P_{\xi}(x)\right]$ and 
\begin{align}\label{Pxi}
P_{\xi}(x):=\min_y & \quad q(\xi)^T y \\
\text { s.t. } & W y \geq d-T x \nonumber\\
& y \in\{0,1\}^n\nonumber
\end{align}
for every $x \in \{0,1\}^n$.
 In the second-stage problem \eqref{Pxi}, only the objective vector 
$q(\xi)$ is stochastic. 

Benders decomposition  (see e.g. \cite{geoffrion1972generalized,rahmaniani2017Benders}) is an efficient cutting-plane approach for solving \eqref{BKP} which approximates $P(\cdot)$ by pointwise maximum of cuts for $P(\cdot)$.
Specifically, an affine function $A(\cdot)$ such that $P(x') \ge {\cal A}(x')$ for every $ x' \in  \{0,1\}^n$
is called a cut for $P(\cdot)$;
moreover, a cut for $P(\cdot)$
is tight at $x$ if $P(x)={\cal A}(x)$.
Benders decomposition starts with a cut ${\cal A}_0$ for $P(\cdot)$ and computes a sequence $\{x_k\}$ of iterates as follows: given cuts $\{{\cal A}_{i}(\cdot)\}_{i=0}^{k-1}$ for $P(\cdot)$,
it computes $x_k$  as 
\begin{align}\label{appint}
x_{k}= \argmin_x  & \ c^T x+P_{k}(x) \\
\text { s.t. } & A x \geq b \nonumber \\
& x \in\{0,1\}^n  \nonumber
\end{align}
where
\[
P_k (\cdot) = \max_{i=0,\cdots,k-1} {\cal A}_{i}(\cdot);
\]
it then uses $x_k$ to generate 
a new cut ${\cal A}_k$ for $P(\cdot)$ and repeats the above steps with $k$ replaced by  $k+1$.
Problem \eqref{appint} can be
easily 
formulated as an equivalent linear integer programming problem. 

We  now describe how to generate a Lagrangian cut for $P(\cdot)$ using a given point  $x \in \{0,1\}^n$.
First, for every $\xi \in \Xi$, \eqref{Pxi} is equivalent to
$$
\begin{aligned}
\min_{y,u} & \quad q(\xi)^T y \\
\text { s.t. } & W y + Tu \geq d \\
& y \in\{0,1\}^n, \, u \in[0,1]^n\\
&u-x=0.
\end{aligned}
$$
By dualizing the constraint $u-x=0$, we obtain the Lagrangian dual (LD) problem
\begin{equation} \label{generalx'}
    D_\xi(x) :=\max_{\pi} L_\xi(x;\pi) 
\end{equation}
where 
\begin{align}\label{generalx}
L_\xi(x;\pi) &:= \min_{y,u} \, q(\xi)^Ty - \pi^T(u-x)  \\
&\text{s.t.}  \quad Wy+Tu \ge d \nonumber\\
          & \  \qquad y \in \{0,1\}^n, \, u \in [0,1]^n \nonumber.
\end{align}

Let $\pi_\xi(x)$ denote an optimal solution of \eqref{generalx'}.
The optimal values $P_\xi(\cdot)$ and $D_\xi(\cdot)$ of  \eqref{Pxi} and \eqref{generalx'}, respectively, are known to satisfy the following two properties
for every $\xi \in \Xi$ and $x \in \{0,1\}^n$:
\begin{itemize}
    \item [(i)]
 $ D_\xi(x') \ge D_\xi(x) + \inner{\pi_\xi(x)}{x'-x}$ for  every 
 $x' \in \{0,1\}^n$;
 \item [(ii)]
    $P_{\xi}(x) =  D_\xi(x)$.
\end{itemize}
Property (i) and the inequality
$P_\xi(x) \ge D_\xi(x)$ for every $x \in \{0,1\}^n$ can be found in many textbooks dealing with Lagrangian duality theory.
On the other hand,  property (ii) has been established in \cite{zou2019stochastic} by exploiting the special binary structure of \eqref{Pxi}.
Defining 
\[
\pi(x):=\E[\pi_{\xi}(x)]
\]
and taking expectation of the relations in (i) and (ii), we easily see  that
\[
P(x') \ge P(x) + \inner{\pi(x)}{x' - x} \quad \forall x' \in \{0,1\}^n,
\]
and hence that ${\cal A}_x (\cdot) :=
P(x) + \inner{\pi(x)}{\cdot - x}$ is a tight cut for $P(\cdot)$ at $x$.

Computation of  $P(x)$
assumes that the optimal value $P_\xi(\cdot)$  
of
  \eqref{Pxi} can be efficiently computed for every $\xi \in \Xi$. 
Computation of $\pi(x)$
assumes that an optimal solution of \eqref{generalx'} can be computed for every $\xi \in \Xi$.
Noting that \eqref{generalx'} is an unconstrained convex nonsmooth optimization problem in terms of variable $\pi$ and its optimal value $D_{\xi}(x)$  is the (already computed) optimal value $P_{\xi}(x)$  of \eqref{Pxi},
we use the  Ad-GPB*  to
obtain a near optimal solution
$\approx \pi_\xi(x)$ of \eqref{generalx'}.
 For the purpose of this subsection,
 we use several instances of  \eqref{generalx'} to benchmark the methods described at the beginning of this section. 



For every $(\xi,x) \in \Xi \times \{0,1\}^n$, recall that using Ad-GPB*  to solve \eqref{generalx'}  requires the ability to evaluate $L_\xi(x,\cdot)$ and compute a subgradient of $-L_\xi(x,\cdot)$ at every $\pi \in \R^n$. The value $L_\xi(x,\pi)$ is evaluated  by solving MILP \eqref{generalx}. Moreover, if $(u_\xi(x;\pi),y_\xi(x;\pi))$ denotes an optimal solution of \eqref{generalx}, then $u_\xi(x;\pi)-x$ yields a subgradient of $-L_\xi(x,\cdot)$
at $\pi$.
It is worth noting \eqref{generalx'} is a non-smooth convex problem that does not seem to be tractable by the methods discussed in the papers (see e.g. \cite{chen2014optimal,he2015accelerating,he2016accelerated,kolossoski2017accelerated,monteiro2011complexity,nemirovski2004prox,nemirovski1978cesari,nesterov2005smooth})  for solving
min-max smooth convex-concave saddle-point problems, mainly due to the integrality condition imposed on the decision variable $y$ in \eqref{generalx}.

 Next we describe how the data of \eqref{BKP} and \eqref{Pxi} is generated. We generate three random instances of \eqref{BKP},  each following the same methodology as in \cite{angulo2016improving}.  We set $n=240$ and
 $\Xi=\{1,\ldots,20\}$ with each scenario $\xi \in \Xi$ being equiprobable. We generate matrices $A_1, A_2 \in \R^{50\times 120}$,  $T_1, W \in \R^{5\times 120}$, and vector $c \in \R^{240}$, with all entries i.i.d.\ sampled from the uniform distribution over the integers $\{1, \ldots, 100\}$. We then set $A = [A_1 \,  A_2]$ and $T= [T_1 \, 0]$ where the zero block of $T$ is $5\times 120$. Twenty vectors $\{q(\xi)\}_{\xi \in \Xi}$ with components i.i.d.\ sampled from $\{1,\ldots,100\}$ are generated.
 Finally, we set $b=3(A_1 \mathbf{1}+A_2 \mathbf{1})/4$ and $d=3(W \mathbf{1}+T_1 \mathbf{1})/4$ where $\mathbf{1}$ denotes the  vector of ones.




For each randomly generated  instance of \eqref{BKP}, we run Benders decomposition started from $x_0=\mathbf{1}$ to obtain three iterates $x_1$, $x_2$, and $x_3$.
Each $P(x_k)$ and $\pi(x_k)$ for $k=1,2,3$ are computed
using the twenty randomly generated vectors
$\{q(\xi)\}_{\xi \in \Xi}$
as described above, and hence each iteration solves twenty  LD subproblems as in \eqref{generalx'}. Hence, each randomly generated  instance  of  \eqref{BKP} yields a total of sixty LD instances as in \eqref{generalx'}. The total time to solve these sixty LD instances  is given in Table \ref{tabfirstpb:generalx}
for the three instances of \eqref{BKP}  (named $I1$, $I2$ and $I3$ in the table) and all the benchmarked methods
considered in this section.  In this comparison,
 Ad-GPB*, P-Ad-GPB* and PB* set the initial stepsize $\hat \lam_0$ to $\lam_{\rm pol}(\pi_0)$ where the entries of $\pi_0$ are i.i.d.\ generated from the uniform distribution in $(0,1)$.
 
  All the methods tested are terminated based on the following criterion: \begin{equation}\label{abs-var}
D_\xi(x) - L_\xi(x;\pi_k)  \le \bar \varepsilon 
\end{equation}
where $\bar \varepsilon = 10^{-4}$.


\begin{table}[H] 
\centering
\begin{tabular}{|c|c|c|c|c|c|c|}
\hline
\multicolumn{1}{|c|}{} 
& {Ad-GPB*}
& {P-Ad-GPB*}
& {P-Sub*}
& {PB*}
& {APL*}
\\
\hline
$I1$ & 123s & \textbf{16s} & 157s & 117s & 1978s \\
\hline
$I2$ & 121s & \textbf{15s} & 184s & 145s & 2300s \\
\hline
$I3$ & 151s & \textbf{13s} & 200s & 143s & 2632s \\
\hline
\end{tabular}
\caption{Numerical results for solving LD subproblems. Stopping criterion \eqref{abs-var} with $\bar \varepsilon = 10^{-4}$ is used.}
\label{tabfirstpb:generalx}
\end{table}
 Table \ref{tabfirstpb:generalx} shows that, similar to Table \ref{tabfirstpb:sparse2},  P-Ad-GPB*  outperforms Ad-GPB*, which in turn outperforms APL*.

\subsection{Constrained $l_1$ problem}\label{cl1}
This subsection reports the computational results  of Ad-GPB against APL on a  constrained $l_1$ feasibility problem. 


The constrained $l_1$ feasibility problem consists of
\begin{align}
     \min f(x)&:=\|Ax-b\|_{1}   \\
       \textbf{s.t. } &e^T x \le D \nonumber\\
     & x \ge 0 \nonumber
\end{align}
where $A \in \R^{m\times n}$, $b\in \R^{ m}$, and $D \in \R_{++}$ are known. 
 Matrix $A$ is randomly generated in the form  $A=E N$ where the nonzero entries of the sparse matrix $N \in \R^{m\times n} $ are i.i.d.\ sampled from the standard normal $\mathcal{N}(0,1)$ 
distribution and $E$ is a diagonal matrix where the diagonal of $E$ are i.i.d.\ sampled from $\mathcal{U}[0,1000]$. 
Vector $b$ is randomly generated in the form $b=(b_0)^2$ where the entries of $b_0$  are i.i.d. sampled from the standard Normal distribution $\mathcal{N}(0,1)$. Finally, we generated  $x_0 = (D/n) \mathbf{1}$.   
 The two benchmark methods are tested on eighteen randomly generated sparse constrained $l_1$ feasibility problems
 and are terminated 
 when the  primal-dual criterion
\begin{equation}\label{PD-var}
    f(x_k)-\hat \ell_k \le \bar \varepsilon (f(x_0)-\hat \ell_0)
\end{equation}
is satisfied,
where $\hat \ell_k$ is a lower bound on $\phi_*$ and $\hat \ell_0:= \min \{ \tilde \ell_f(x;x_0) : e^Tx \le D, \, x\ge 0\}$.

We now describe some details about all the tables that appear in this subsection.  We set the target $\bar \varepsilon$ in \eqref{PD-var} as $10^{-3}$ . 
The quantities $\theta_m$, $\theta_n$ and $\theta_s$  are defined as in Table \ref{tabfirstpb:sparse1}. Each table entry corresponds to the CPU running time in seconds. An entry marked $*/N$ indicates that the computed solution failed to achieve the desired relative accuracy within the time limit, in which case $N$ denotes the relative accuracy of the computed solution as defined in \eqref{PD-var}. Bold numbers highlight the method achieving the best performance for each instance.

\renewcommand{\arraystretch}{1}
\begin{table}[H] 
\centering
\footnotesize
\begin{tabular}{|c|c|c||c|c|c|}
\hline
{$(\theta_m,\theta_n ,\theta_s,D)$} 
& {Ad-GPB} & {APL} 
& {$(\theta_m,\theta_n ,\theta_s,D)$} 
& {Ad-GPB} & {APL}\\
\hline

$(1,20,10^{-2},100)$ 
& \textbf{24} & $*/.97e{+00}$
& $(1,20,10^{-2},50)$ 
& \textbf{83} & $*/.80e{+00}$\\
\hline

$(1.5,15,10^{-2},100)$ 
& \textbf{254} & $*/.97e{+00}$
& $(1.5,15,10^{-2},50)$ 
& \textbf{420} & $*/.97e{+00}$\\
\hline

$(2,20,10^{-2},100)$ 
& \textbf{45} & $*/.98e{+00}$
& $(2,20,10^{-2},50)$ 
& \textbf{98} & $*/.85e{+00}$\\
\hline

$(2.5,25,10^{-2},100)$ 
& \textbf{460} & $*/.98e{+00}$
& $(2.5,25,10^{-2},50)$ 
& \textbf{463} & $*/.97e{+00}$\\
\hline

$(3,30,10^{-2},100)$ 
& \textbf{388} & $*/.96e{+00}$
& $(3,30,10^{-2},50)$ 
& \textbf{221} & $*/.65e{+00}$\\
\hline

$(3.5,35,10^{-2},100)$ 
& \textbf{135} & $*/.96e{+00}$
& $(3.5,35,10^{-2},50)$ 
& \textbf{343} & $*/.96e{+00}$\\
\hline

$(4,40,10^{-2},100)$ 
& \textbf{682} & $*/.80e{+00}$
& $(4,40,10^{-2},50)$ 
& \textbf{323} & $*/.54e{+00}$\\
\hline

$(4.5,45,10^{-2},100)$ 
& \textbf{958} & $*/.92e{+00}$
& $(4.5,45,10^{-2},50)$ 
& \textbf{836} & $*/.80e{+00}$\\
\hline

$(5,50,10^{-2},100)$ 
& \textbf{1288} & $*/.76e{+00}$
& $(5,50,10^{-2},50)$ 
& $*/.18e{-02}$ & $*/.51e{+00}$\\
\hline

\end{tabular}
\caption{Numerical results for sparse instances. Stopping criterion \eqref{PD-var} with $\bar \varepsilon = 10^{-3}$ is used and a time limit of
1800 seconds (half an hour) is given. Each table entry corresponds to the CPU running time in seconds. An entry marked $*/N$ indicates that the computed solution failed to achieve the desired relative accuracy within the time limit, in which case $N$ denotes the relative accuracy of the computed solution as defined in \eqref{PD-var}.}
\label{tabfirstpb:sparse4}
\end{table}

Table \ref{tabfirstpb:sparse4} shows that Ad-GPB significantly outperforms APL.

\section{Concluding remarks}\label{CR}



This paper presents a parameter- and line-search-free adaptive proximal bundle method built around two key ingredients: (i) an adaptive strategy for selecting variable proximal step sizes tailored to the problem instance, and (ii) an adaptive cycle-stopping criterion designed to improve the effectiveness of serious steps. Computational experiments show that the proposed method substantially reduces the number of consecutive null steps, leading to shorter cycles, while keeping the number of serious steps at a moderate level. Consequently, it requires fewer iterations than methods based on a fixed proximal step size and a non-adaptive cycle-termination criterion. Moreover, the method is shown to be robust with respect to the choice of the initial stepsize.



We finally discuss some possible extensions of our analysis in this paper. 
The first two extensions are related to the removal of assumption (A1), i.e.,
that $\dom h$ is bounded.
First, it would be interesting to extend the analysis of Rs-Ad-GPB*
without assuming (A1) and establish an $\tilde {\cal O}(1/(\mu \varepsilon))$ iteration-complexity as in Theorem~\ref{thm:strongconvex}.
Second, establishing the iteration complexity of Ad-GPB, or restart variants of it, without assuming (A1)
is a challenging and interesting research direction.


Another potential extension is to analyze the complexity of P-Ad-GPB* described in the second paragraph of 
\S\ref{numer}, as it performed well in our computational experiments. 



Finally, our current analysis does not cover the one-cut bundle update scheme
(see Subsection~3.1 of \cite{liang2024unified}), since this scheme is not a
special case of \eqref{def:Gamma}, as noted in the remark at the end of
\S\ref{sec:shadow}. It would therefore be interesting to extend our analysis
to variants of Ad-GPB* and/or Ad-GPB based on the one-cut bundle update scheme.

\bibliographystyle{plain}
\bibliography{bundle}

\end{document}